\pdfoutput=1
\documentclass[11pt]{amsart}
\usepackage{amsmath}
\usepackage{amssymb}
\usepackage{harpoon}
\usepackage[pdftex]{graphicx}
\usepackage{esint}
\usepackage[latin1]{inputenc}
\usepackage{xcolor}

\theoremstyle{plain}
\newtheorem{theorem}{Theorem}[section]
\newtheorem{lemma}[theorem]{Lemma}
\newtheorem{prop}[theorem]{Proposition}

\theoremstyle{definition}
\newtheorem{remark}[theorem]{Remark}

\numberwithin{equation}{section}

\theoremstyle{plain}

\numberwithin{equation}{section}

\allowdisplaybreaks

\begin{document}


\title[Harmonic Functions and Mass]{Harmonic Functions and The Mass of 3-Dimensional Asymptotically Flat Riemannian Manifolds}

\author[Bray]{Hubert L. Bray}
\address{
Department of Mathematics\\
Duke University\\
Durham, NC, 27708\\
USA}
\email{bray@math.duke.edu}

\author[Kazaras]{Demetre P. Kazaras}
\author[Khuri]{Marcus A. Khuri}
\address{
Department of Mathematics\\
Stony Brook University \\
Stony Brook, NY, 11794-3660\\
USA}
\email{demetre.kazaras@stonybrook.edu, khuri@math.sunysb.edu}

\author[Stern]{Daniel L. Stern}
\address{
Department of Mathematics\\
University of Toronto\\
Toronto, ON, M5S 2E4\\
Canada
}
\email{dl.stern@utoronto.ca}

\thanks{D. Kazaras acknowledges the support of NSF Grant DMS-1547145. M. Khuri acknowledges the support of NSF Grant DMS-1708798.}

\begin{abstract}
An explicit lower bound for the mass of an asymptotically flat Riemannian 3-manifold is given in terms of linear growth harmonic functions and scalar curvature. As a consequence, a new proof of the positive mass theorem is achieved in dimension three.
The proof has parallels with both the Schoen-Yau minimal hypersurface technique and Witten's spinorial approach. In particular, the role of harmonic spinors and the Lichnerowicz formula in Witten's argument is replaced by that of harmonic functions and a formula introduced by the fourth named author in recent work, while the level sets of harmonic functions take on a role similar to that of the Schoen-Yau minimal hypersurfaces.
\end{abstract}

\maketitle

\section{Introduction}
\label{sec1} \setcounter{equation}{0}
\setcounter{section}{1}

Let $(M,g)$ be a smooth connected 3-dimensional asymptotically flat Riemannian manifold with nonnegative scalar curvature $R_g \geq 0$. The notion of asymptotic flatness means that there is a compact set $\mathcal{K}\subset M$ such that $M\setminus \mathcal{K}=\cup_{k=1}^{k_0}M_{end}^{k}$ where the ends $M_{end}^k$ are pairwise disjoint and diffeomorphic to the complement of a ball $\mathbb{R}^3 \setminus B_1$, and there exists in each end a coordinate system satisfying
\begin{equation}\label{af}
|\partial^l (g_{ij}-\delta_{ij})(x)|=O(|x|^{-q-l}),
\end{equation}
for some $q>\tfrac{1}{2}$ and with $l=0,1,2$. The scalar curvature is assumed to be integrable $R_g \in L^1(M)$ so that the ADM mass of each end is well-defined \cite{Bartnik} and given by
\begin{equation}
m=\lim_{r\rightarrow\infty}\frac{1}{16\pi}\int_{S_{r}}\sum_i (g_{ij,i}-g_{ii,j})\upsilon^j dA,
\end{equation}
where $\upsilon$ is the unit outer normal to the coordinate sphere $S_r$ of radius $r=|x|$ and $dA$ denotes its area element. The positive mass theorem asserts that this parameter has a sign, and it characterizes Euclidean space as the unique manifold in this class with vanishing mass.

\begin{theorem}\label{pmt}
If $(M,g)$ is complete and asymptotically flat with nonnegative scalar curvature then $m\geq 0$, and $m=0$ if and only if $(M,g)\cong(\mathbb{R}^3, \delta)$.
\end{theorem}

This theorem was first established in the late 1970's by Schoen and Yau \cite{SchoenYauI,SchoenYauIII} via a contradiction argument, and is based on the existence of stable minimal hypersurfaces along with manipulations of the stability inequality. Shortly after this Witten \cite{ParkerTaubes,Witten} found an alternate proof in which the mass is expressed as a sum of squares. This proof relies on the existence of harmonic spinors and the Lichnerowicz formula. More recently two other proofs have been given in the general case. One by Huisken and Ilmanen \cite{HuiskenIlmanen}, which arose out of their study of the Penrose inequality, follows from the existence of a weak version of inverse mean curvature flow and  monotonicity of Hawking mass. The other is a Ricci flow proof and is due to Li \cite{Li}. Further proofs have been given in special cases, such as that of Brill \cite{Brill} in the axisymmetric setting. It should also be noted that Lohkamp \cite{Lohkamp} showed how the positive mass theorem can be reduced to the nonexistence of positive scalar curvature metrics on the connected sum $N\# T$ of a compact manifold $N$ with a torus $T$. See \cite{Lee} for a survey of these topics. Furthermore, we point out the articles of Schoen and Yau \cite{SchoenYauIV} and Lohkamp \cite{Lohkamp1} which address the higher dimensional case.

The purpose of the current article is to give an explicit lower bound for the mass in terms of linear growth harmonic functions and scalar curvature. This approach is based on an integral inequality due to Stern \cite{Stern}, and leads to a new and relatively simple proof of Theorem \ref{pmt}. Associated with each asymptotic end $M_{end}$ there is a corresponding \textit{exterior region} $M_{ext} \supset M_{end}$, which is diffeomorphic to the complement of a finite number of balls (with disjoint closure) in $\mathbb{R}^3$ and has minimal boundary \cite[Lemma 4.1]{HuiskenIlmanen}.

\begin{theorem}\label{mass.bd}
Let $(M_{ext},g)$ be an exterior region of a complete asymptotically flat Riemannian 3-manifold $(M,g)$ with mass $m$. Let $u$ be a harmonic function on $(M_{ext},g)$ satisfying Neumann boundary conditions at $\partial M_{ext}$, and which is asymptotic to one of the asymptotically flat coordinate functions of the associated end.
Then
\begin{equation}\label{masslowerb}
m \geq \frac{1}{16\pi} \int_{M_{ext}}\left(\frac{|\nabla^2 u|^2}{|\nabla u|}+R_g |\nabla u|\right) dV.
\end{equation}
In particular, if the scalar curvature is nonnegative then $m\geq 0$. Furthermore, if $m=0$ then $(M,g)\cong(\mathbb{R}^3,\delta)$.
\end{theorem}

Two approaches to the positive mass theorem will be presented within the context of the harmonic level set technique. They differ in their handling of the exterior region boundary, and in their use of the asymptotically flat geometry. In the first method Neumann boundary conditions are imposed on $\partial M_{ext}$ in order to deal with boundary terms appearing in Stern's integration formula \cite{Stern}, while in the second method the Mantoulidis neck construction \cite{Mantoulidis} is used to cap-off the boundary spheres so that the resulting manifold is diffeomorphic to $\mathbb{R}^3$ and still possesses nonnegative scalar curvature. Within the asymptotic end harmonic coordinates are employed along with cylindrical domains in the first approach to extract the mass and compute total geodesic curvatures. On the other hand, the second approach utilizes a density theorem to reduce the asymptotic geometry to that of Schwarzschild where the analysis is then performed on coordinate spheres.



\section{Preparing the Data}
\label{sec2} \setcounter{equation}{0}
\setcounter{section}{2}

Within the context of the $3$-dimensional positive mass theorem, simplifications of the asymptotics and topology may be assumed without loss of generality. More precisely Schoen and Yau \cite{SchoenYauIII} showed that metrics with harmonic asymptotics are dense in the relevant class of metrics, and Bray \cite{Bray} (see also \cite[Proposition 3.3]{CorvinoPollack}) extended this to show that in fact harmonic asymptotics may be replaced with Schwarzschild asymptotics. As for the topology of $M$, one may consider the portion of $M$ outside the outermost minimal surface, and fill in the resulting spherical holes using work of Mantoulidis \cite{Mantoulidis} and the Miao smoothing \cite{Miao}. This procedure allows one to reduce the topology of $M$ to $\mathbb{R}^3$. It should be noted that this reduction is specific to dimension 3.

\begin{prop}\label{simplify}
Let $(M,g)$ be a smooth 3-dimensional complete asymptotically flat Riemannian manifold having nonnegative scalar curvature $R_g \geq 0$, and with mass $m$ of a designated end $M_{end}^+$. Given $\varepsilon>0$, there exists a smooth $3$-dimensional complete asymptotically flat Riemannian manifold $(\overline{M},\overline{g})$ with nonnegative scalar curvature $R_{\overline{g}}\geq 0$ and satisfying the following properties.
\begin{enumerate}
\item The underlying manifold $\overline{M}$ is diffeomorphic to $\mathbb{R}^3$.

\item The mass $\overline{m}$ of the single end $\overline{M}_{end}$ satisfies $|m-\overline{m}|<\varepsilon$.

\item In the asymptotic coordinates of $\overline{M}_{end}$, $\overline{g}=\left(1+\tfrac{\overline{m}}{2r}\right)^4 \delta$.
\end{enumerate}
\end{prop}

\begin{proof}
By passing to the orientable double cover if necessary we may assume that $M$ is orientable. Moreover
by applying an appropriate conformal deformation with conformal factor approximating 1, such that the deformed mass differs from the original by an arbitrarily small amount, we may assume that $(M,g)$ has positive scalar curvature $R_{g}>0$ everywhere. Let $\mathcal{S}\subset M$ denote the \textit{trapped region} in the sense of \cite{HuiskenIlmanen}. If $\mathcal{S}=\emptyset$ then $M$ is diffeomorphic to $\mathbb{R}^3$ \cite[Lemma 4.1]{HuiskenIlmanen}, and we set $\overline{M}=M$. The desired metric $\overline{g}$ then arises from Bray's density result \cite[Proposition 3.3]{CorvinoPollack}. If $\mathcal{S}\neq \emptyset$, then
consider $\overline{M}^+$, the metric closure of the component of $M\setminus\mathcal{S}$ containing $M_{end}^+$.
According to \cite[Lemma 4.1]{HuiskenIlmanen} this exterior region is diffeomorphic to the complement of a finite union of balls, so that
\begin{equation}
\overline{M}^+ =\mathbb{R}^3\setminus \cup_{i=1}^k B_i,
\end{equation}
where the spheres $\cup_i S^2_i=\cup_i \partial B_i$ are homologically area outer-minimizing. Since each submanifold $(S^2_i,\gamma_i=g|_{S_i^2})$ is a $2$-sided stable minimal surface in an ambient space of positive scalar curvature, the principal eigenvalue of $-\Delta_{\gamma_i}+K_{\gamma_i}$ is positive, where $K_{\gamma_i}$ denotes Gaussian curvature. The hypotheses of \cite[Corollary 2.2.13]{Mantoulidis} are then satisfied, so that for each $i=1,\dots,k$ there is a Riemannian $3$-ball $(D_i,g_i)$ with positive scalar curvature, minimal boundary, and satisfying $\partial(D_i,g_i)\cong (S^2_i,\gamma_i)$. Glue in these $3$-balls to form
\begin{equation}
\overline{M}=\overline{M}^+\bigcup_{\mathcal{S}}\left(\cup_{i=1}^k D_i\right),
\end{equation}
and equip $\overline{M}$ with a $C^{0,1}$-Riemannian metric that agrees with $g$ on $\overline{M}^+$ and $g_i$ on each $D_i$, see Figure \ref{pic:lemma2-1}. Next, smooth a tubular neighborhood of $\mathcal{S}$ followed by a conformal deformation as in \cite[Sections 3 \& 4]{Miao}, to obtain an asymptotically flat metric $\tilde{g}$ on $\overline{M}$ with nonnegative scalar curvature and mass satisfying $|m-\tilde{m}|<\varepsilon/2$. Now apply Bray's density result \cite[Proposition 3.3]{CorvinoPollack} to $\tilde{g}$ to produce the desired metric $\overline{g}$ with mass $\overline{m}$ satisfying $|\tilde{m}-\overline{m}|<\varepsilon/2$.
\end{proof}

\begin{remark}\label{rem1}
In Proposition \ref{simplify}, the conclusion that $\overline{M}$ is diffeomorphic to $\mathbb{R}^3$ relies on deep results outlined in \cite{HuiskenIlmanen}. In light of this, it is worth pointing out that ultimately we do not require the full strength of this topological simplification. Indeed, the only time this portion of Proposition \ref{simplify} is utilized, is in the proof of Lemma \ref{lharmonic} where only the triviality of $H_2(\overline{M};\mathbb{Z})$ is needed. This weaker simplification may be achieved via more elementary means.  By following the arguments of \cite[page 140]{Lee}, there exists $\widetilde{M}^+$ containing $M_{end}^+$ whose boundary consists of minimal spheres and satisfies $H_2(\widetilde{M}^+ ,\partial\widetilde{M}^+;\mathbb{Z})=0$. Then filling in with discs as above yields the desired conclusion.
%
\end{remark}

\begin{figure}
\begin{picture}(0,0)
\put(105,150){$\mathcal{S}$}
\put(103,153){\vector(-1,0){35}}
\put(113,153){\vector(1,0){50}}
\put(280,130){\Large{$(M,g)$}}
\put(280,50){\Large{$(\overline{M},\overline{g})$}}
\thicklines
\put(130,82){\vector(0,-1){40}}
\put(115,82){\vector(0,-1){40}}
\end{picture}
\includegraphics[scale=.6]{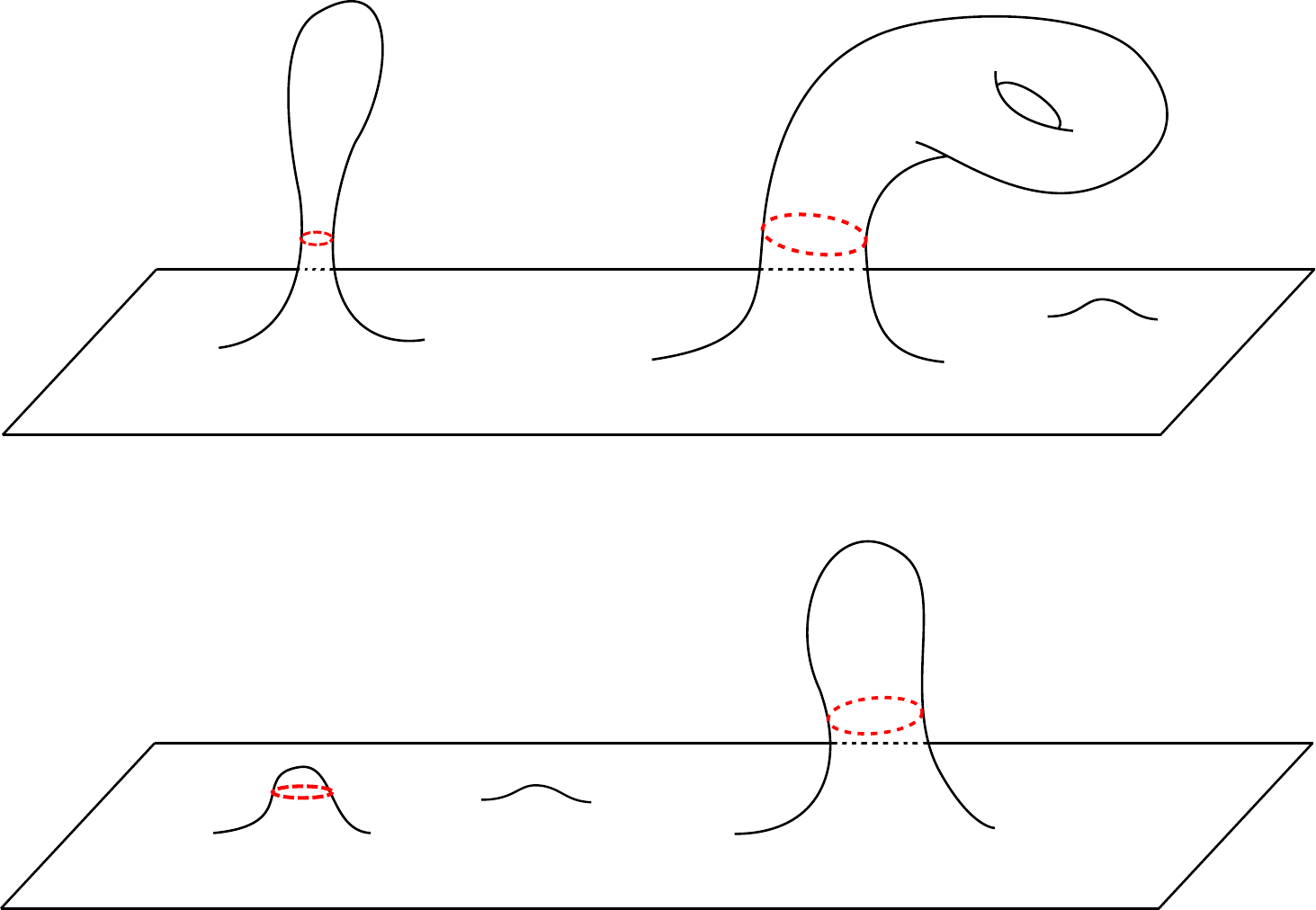}
\caption{A schematic description of the construction in Proposition \ref{simplify}.}\label{pic:lemma2-1}
\end{figure}
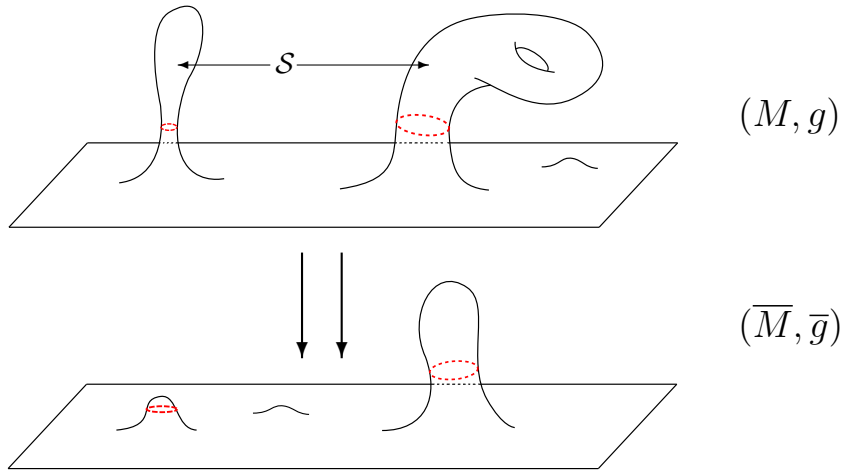

A 3-dimensional Riemannian manifold satisfying points (1) and (3) of Proposition \ref{simplify} will be referred to as \textit{Schwarzschildian}. This proposition allows the proof of Theorem \ref{pmt} to be reduced to the following Schwarzschildian case, which will be established in Section \ref{sec5}.

\begin{theorem}\label{scpmt}
If $(M,g)$ is complete and Schwarzschildian with nonnegative scalar curvature then $m\geq 0$, and $m=0$ if and only if $(M,g)\cong(\mathbb{R}^3, \delta)$.
\end{theorem}

\section{Linear Growth Harmonic Functions}
\label{sec3} \setcounter{equation}{0}
\setcounter{section}{3}

\subsection{Harmonic functions on Schwarzschildian ends}
Suppose that $(M,g)$ is Schwarzschildian, and in $M_{end}$ write $g =w^4 \delta$ where $w=1+\tfrac{m}{2r}$. Let
\begin{equation}
L_{g}=\Delta_g -\frac{1}{8}R_g
\end{equation}
be the conformal Laplacian. According to the conformal invariance of this operator,
\begin{equation}
L_{g}v=w^{-5}L_{\delta}(w v)
\end{equation}
for any function $v$. Let $\ell(x)=a_i x^i$ be a linear function in the asymptotically flat coordinates $\{x^i\}_{i=1}^3$ on $M_{end}$. Since $R_g\equiv 0$ in $M_{end}$ it follows that
\begin{equation}\label{wl}
\Delta_{g}(\ell w^{-1})=L_{g}(\ell w^{-1})=w^{-5}L_{\delta}\ell
=w^{-5}\Delta_{\delta}\ell=0.
\end{equation}
We can now find harmonic functions on $M$ with the following prescribed linear asymptotics. Given $a_i$ there exists a constant $a$ such that
\begin{equation}\label{uequation}
\begin{cases}
\Delta_{g}u=0 &\text{ on } M ,\\
u(x)=\frac{a_i x^i}{1+\tfrac{m}{2r}}+\frac{a}{r}+O_2(r^{-2}) & \text{ in } M_{end},
\end{cases}
\end{equation}
where the notation $v=O_l(r^{-k})$ asserts that $|\partial^j v|\leq Cr^{-k-j}$ for $j\leq l$. To see this, let $u_0\in C^{\infty}(M)$ be any smooth function satisfying $u_0\equiv\ell w^{-1}$ in $M_{end}$, and set $f=-\Delta_{g}u_0$. Notice that \eqref{wl} implies $f\equiv 0$ in $M_{end}$. By a standard argument \cite[Lemma 3.2]{SchoenYauI},
there exists a function $u_1\in C^{\infty}(M)$ solving
\begin{equation}
\begin{cases}
\Delta_{g}u_1=f &\text{ on } M ,\\
u_1(x)=\frac{a}{r}+O_2(r^{-2}) & \text{ in } M_{end},
\end{cases}
\end{equation}
for some constant $a$. The desired unique solution of \eqref{uequation} is $u=u_0 +u_1$.

\begin{lemma}\label{lharmonic}
Let $(M,g)$ be complete and Schwarzschildian. For any linear function $\ell$ in the coordinates of $M_{end}$, there exists a unique solution $u_\ell$ of \eqref{uequation}. Moreover, all regular level sets of  $u_\ell$ are connected and noncompact with a single end modeled on $\mathbb{R}^2\setminus B_1$.
\end{lemma}

\begin{proof}
The discussion preceding the lemma establishes the existence of the solutions $u_\ell$, and uniqueness follows from the maximum principle. For such an asymptotically linear harmonic function $u_\ell$, let $t$ be a regular value of $u_\ell$ and consider the level set $\Sigma_t=u^{-1}(t)$. Suppose that there is a compact connected component $\Sigma'_{t}\subset\Sigma_t$. Notice that $\Sigma'_{t}$ is a properly embedded submanifold and is $2$-sided (has trivial normal bundle). Since $M=\mathbb{R}^3$ has trivial homology, $\Sigma'_{t}$ must bound a compact region of $M$. By uniqueness of solutions to the Dirichlet problem for harmonic functions, $u_\ell\equiv t$ on this region. However this contradicts the assumption that $t$ is a regular value. It follows that all components of $\Sigma_t$ are noncompact. Furthermore, since it is properly embedded $\Sigma_t$ is a closed subset of $M$. Thus if any component of $\Sigma_t$ stays within $M_r$, the compact region bounded by the coordinate sphere $S_r\subset M_{end}$, it must be compact which is a contradiction. We conclude that each component must extend outside $S_r$ for all $r$.

\begin{figure}
\begin{picture}(0,0)
\put(210,29){{\color{blue}{$\Sigma_1$}}}
\put(210,48){{\color{red}{$\Sigma_2$}}}
\put(210,66){$\Sigma_3$}
\end{picture}
\includegraphics[scale=.6]{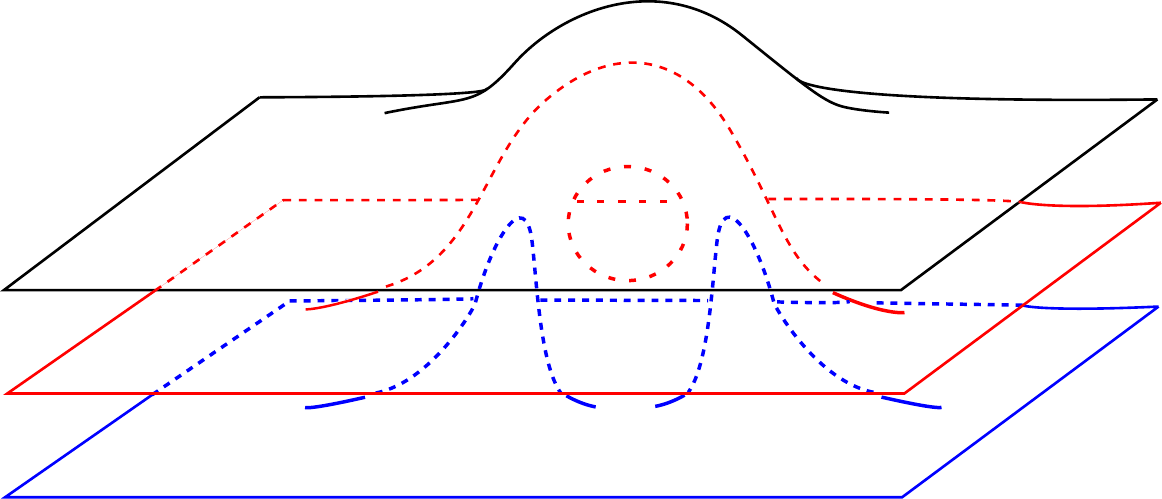}
\caption{Possible level sets of a harmonic function from Lemma  \ref{lharmonic}.}\label{pic:lemma3-1}
\end{figure}
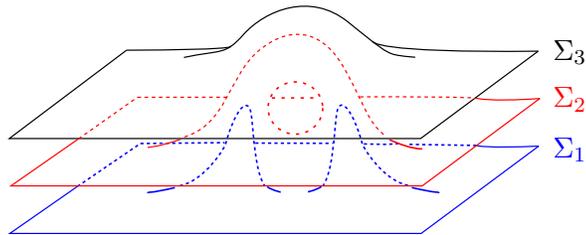

The asymptotics of $u_{\ell}$ imply that there exists a constant $C$, such that for all sufficiently large $r$ the level set $\Sigma_t$ lies within the slab $\{x\in M\setminus M_r \mid t-C<\ell(x)<t+C\}$. More precisely, the implicit function theorem shows that $\Sigma_t$ is represented uniquely in this region as a graph over the plane $t=\ell(x)$. It follows that $\Sigma_t$ is connected and has a single end modeled on $\mathbb{R}^2\setminus B_1$.
\end{proof}

\subsection{Harmonic coordinates}\label{sec:harmoniccoords}

In the general case of an asymptotically flat 3-manifold $(M,g)$, not necessarily Schwarzschildian, consider the exterior region $M_{ext}$ associated with a given end $M_{end}$. Let $y^i$, $i=1,2,3$ denote the given asymptotically flat coordinate system in $M_{end}$. The analysis of \cite[Theorem 3.1]{Bartnik} may be appropriately modified in order to produce harmonic coordinates satisfying Neumann boundary conditions. That is, there exist functions $x^i \in C^{\infty}(M_{ext})$ satisfying
\begin{equation}\label{e:harmdecay0}
\Delta_g x^i=0 \quad\text{ on }\quad M_{ext}, \quad \partial_{\upsilon}x^i=0\quad\text{ on }\quad\partial M_{ext},\quad
|x^i-y^i|=o(|y|^{1-q})\quad\text{ as }\quad |y|\rightarrow\infty,
\end{equation}
where $q$ is the order of asymptotically flat decay in \eqref{af}. This decay is still valid for the harmonic coordinates, that is
\begin{equation}\label{e:harmdecay}
|\partial^l (g_{ij}-\delta_{ij})(x)|=O(|x|^{-q-l}),\quad\quad l=0,1,2.
\end{equation}
Harmonic coordinates are particularly well suited for studying the mass \cite{Bartnik},
and will play an important role in the computation of asymptotic boundary terms appearing in the integral inequalities of Section \ref{sec4} below.

\section{Relating Scalar Curvature to Level Set Geometry}
\label{sec4} \setcounter{equation}{0}
\setcounter{section}{4}

The purpose of this section is to obtain integral inequalities for the scalar curvature of a compact Riemannian manifold equipped with a harmonic function, building on the techniques introduced by the fourth named author in \cite{Stern}. Note that our setting is slightly different from that of \cite{Stern}, which studies closed 3-manifolds with harmonic maps to $S^1$, while we work with harmonic functions on compact manifolds with boundary where additional boundary conditions are needed. As in \cite{Stern}, the first step in obtaining the relevant identities is to apply the Gauss equations to extract scalar curvature on a regular level set of a harmonic function. Note that in the next result, the dimension is not restricted to three.

\begin{lemma}\label{lem:bochner}
Suppose that $(M,g)$ is a Riemannian manifold and $u:M\to \mathbb{R}$ is harmonic with regular level set $\Sigma$. Then, on $\Sigma$, the following identity holds
\begin{equation}\label{e:bochner}
\mathrm{Ric}(\nabla u,\nabla u)=\frac12|\nabla u|^2\left(R_g-R_\Sigma \right)+|\nabla|\nabla u||^2-\frac12|\nabla^2u|^2,
\end{equation}
where $R_g$ and $R_{\Sigma}$ denote the respective scalar curvatures.
\end{lemma}

\begin{proof}
Since $\Sigma$ is a regular level set, its unit normal is given by $\nu=\frac{\nabla u}{|\nabla u|}$. Taking two traces of the Gauss equations then yields
\begin{equation}\label{e:GC}
R_g-2\mathrm{Ric}\left(\frac{\nabla u}{|\nabla u|},\frac{\nabla u}{|\nabla u|}\right)=R_\Sigma+|II|^2-H^2,
\end{equation}
where $H$ and $II$ are the mean curvature and second fundamental form of $\Sigma$. The second fundamental form is given by $II=\frac{\nabla^2_\Sigma u}{|\nabla u|}$,
where $\nabla^2_\Sigma u$ denotes the Hessian of $u$ restricted to $T\Sigma\otimes T\Sigma$. It follows that
\begin{equation}\label{e:sff}
|II|^2=|\nabla u|^{-2}\left(|\nabla^2u|^2-2|\nabla|\nabla u||^2+[\nabla^2u(\nu,\nu)]^2\right),
\end{equation}
and because $u$ is harmonic
\begin{equation}\label{e:meancurv}
H=\mathrm{Tr}_{\Sigma}II=|\nabla u|^{-1}\left(\mathrm{Tr}_g\nabla^2 u-\nabla^2u(\nu,\nu)\right)
=-|\nabla u|^{-1}\nabla^2u(\nu,\nu).
\end{equation}
Combining equations \eqref{e:sff} and \eqref{e:meancurv} produces
\begin{align}\label{e:sff_meancurv}
|II|^2 -H^2=|\nabla u|^{-2}\left(|\nabla^2u|^2-2|\nabla|\nabla u||^2\right).
\end{align}
Inserting this into \eqref{e:GC} gives the desired result.
\end{proof}

The formula of Lemma \ref{lem:bochner} will be combined with Bochner's identity and integrated by parts over a compact manifold with boundary, while applying the coarea formula with harmonic level sets.
For a function $u:\Omega\to\mathbb{R}$ on a compact manifold $\Omega$, let $\overline{u}$ and $\underline{u}$ be the maximum and minimum values of $u$, respectively. The following computation plays a key role in both of our approaches for obtaining lower bounds on the ADM mass. We remark that related computations for $S^1$-valued harmonic maps with homogeneous Neumann condition can be found in the paper \cite{BrayStern}, where several applications to the geometry of compact $3$-manifolds are obtained.

\begin{prop}\label{prop:gen.stern}
Let $(\Omega^3,g)$ be an $3$-dimensional oriented compact Riemannian manifold with boundary decomposed into $\partial\Omega=P_1\sqcup P_2$. Let $u:\Omega\to\mathbb{R}$ be a harmonic function satisfying the Neumann condition $\partial_\upsilon u\equiv 0$ on $P_1$ and the nondegeneracy condition $|\nabla u|_{P_2}|>0$ on $P_2$. Then
\begin{align}\label{bdry.strn.0}
\int_{\underline{u}}^{\overline{u}}\left(\int_{\Sigma_{t}}
\frac{1}{2}\left(\frac{|\nabla^2u|^2}{|\nabla u|^2}+R_g\right)dA+\int_{\partial\Sigma_{t}\cap P_1}H_{P_1} \right)dt&\notag\\
\leq\int_{\underline{u}}^{\overline{u}}\left(2\pi \chi(\Sigma_{t})-\int_{\partial\Sigma_t\cap P_2} \kappa_{\partial\Sigma_t}\right)dt&
+\int_{P_2}\partial_\upsilon|\nabla u|dA,
\end{align}
where $\kappa_{\partial\Sigma_t}$ denotes the geodesic curvature of $\partial\Sigma_t\subset\Sigma_t$, $H_{P_1}$ denotes the mean curvature of $P_1$, and $\upsilon$ is the unit outer normal to $\partial\Omega$. In the case $P_1=\varnothing$, we record also the equivalent formulation
\begin{equation}\label{bdry.strn.1}
\int_{\underline{u}}^{\overline{u}}\int_{\Sigma_t}\frac{1}{2}
\left(\frac{|\nabla^2u|^2}{|\nabla u|^2}+R_g-R_{\Sigma_t}\right)dA dt\leq \int_{\partial\Omega}\partial_{\upsilon}|\nabla u|.
\end{equation}
\end{prop}

\begin{proof}
Let $\varepsilon>0$ and consider $\phi=\sqrt{|\nabla u|^2+\varepsilon}$.
By Bochner's identity
\begin{align}\label{e:phidelta1}
\begin{split}
\Delta_{g}\phi=&\frac{\Delta_{g}|\nabla u|^2}{2\phi}-\frac{|\nabla|\nabla u|^2|^2}{4\phi^3}\\
=&\phi^{-1}\left(|\nabla^2 u|^2+\mathrm{Ric}(\nabla u,\nabla u)-\phi^{-2} |\nabla u|^2|\nabla|\nabla u||^2\right).
\end{split}
\end{align}
It follows that on a regular level set $\Sigma$, Lemma \ref{lem:bochner} may be applied to find
\begin{equation}\label{e:phidelta2}
\Delta_{g}\phi\geq\frac12 \phi^{-1}\left(|\nabla^2 u|^2+|\nabla u|^2(R_g-R_{\Sigma})\right).
\end{equation}

Let $\mathcal{A}\subset [\underline{u},\overline{u}]$ be an open set containing the critical values of $u$, and denote the complementary closed set by $\mathcal{B}\subset[\underline{u},\overline{u}]$. Note that, by virtue of the boundary conditions for $u$, $\mathcal{A}$ also contains all critical values for the restriction $u|_{\partial\Omega}$ of $u$ to the boundary.

Now, integration by parts yields
\begin{equation}\label{e:intbyparts}
\int_{\partial\Omega}\partial_\upsilon \phi dA=\int_{\Omega}\Delta_g \phi dV=\int_{u^{-1}(\mathcal{A})}\Delta_g \phi dV+\int_{u^{-1}(\mathcal{B})}\Delta_g \phi dV.
\end{equation}
In order to control the integral over $u^{-1}(\mathcal{A})$, observe that \eqref{e:phidelta1} and Cauchy-Schwarz give the estimate
\begin{equation}
\Delta_{g} \phi
\geq\phi^{-1}\mathrm{Ric}(\nabla u,\nabla u)
\geq-\parallel\mathrm{Ric}\parallel |\nabla u|.
\end{equation}
Applying the coarea formula to $u:u^{-1}(\mathcal{A})\to\mathcal{A}$ then produces
\begin{equation}\label{e:Apart}
-\int_{u^{-1}(\mathcal{A})}\Delta_g \phi dV\leq\int_{u^{-1}(\mathcal{A})}\parallel\mathrm{Ric}\parallel |\nabla u|dV
\leq C \int_{t\in\mathcal{A}}\mathcal{H}^2(\Sigma_t)dt,
\end{equation}
for some constant $C$ independent of $\varepsilon$ and the choice of $\mathcal{A}$.
In addition, applying the coarea formula to $u:u^{-1}(\mathcal{B})\to\mathcal{B}$ in conjunction with \eqref{e:phidelta2} gives
\begin{equation}
\int_{u^{-1}(\mathcal{B})}\Delta_g \phi dV
\geq\frac{1}{2}\int_{t\in\mathcal{B}}\int_{\Sigma_t}\phi^{-1}\left[\frac{|\nabla^2 u|^2}{|\nabla u|^2}+(R_g-R_{\Sigma_t})\right]dAdt.
\end{equation}
Putting this all together yields
\begin{equation}\label{p0o9}
\frac{1}{2}\int_{t\in\mathcal{B}}\int_{\Sigma_t}\phi^{-1}|\nabla u|\left[\frac{|\nabla^2 u|^2}{|\nabla u|^2}+(R_g-R_{\Sigma_t})\right]dAdt
\leq\int_{\partial\Omega}\partial_\upsilon \phi dA
+C \int_{t\in\mathcal{A}}\mathcal{H}^2(\Sigma_t)dt.
\end{equation}

Next, we employ the homogeneous Neumann condition $\partial_{\upsilon}u\equiv 0$ to rewrite the boundary integral $\int_{P_1}\partial_{\upsilon}\phi$. Indeed, note that, away from critical points of $u$ along $P_1$, we have
\begin{equation}\label{e:boundarytermn}
\partial_\upsilon \phi=\phi^{-1}\langle \nabla_{\nabla u}\nabla u,\upsilon\rangle=-\phi^{-1}\langle \nabla u,\nabla_{\nabla u}\upsilon\rangle,
\end{equation}
where in the last line the Neumann condition was used. Writing $\nu=\frac{\nabla u}{|\nabla u|}$ and continuing to use the homogeneous Neumann condition, a brief calculation shows that
\begin{equation}
\langle \nu, \nabla_{\nu}\upsilon\rangle=II_{\partial \Omega}(\nu,\nu)=H_{P_1}-\kappa_{\partial\Sigma_t},
\end{equation}
so we can rewrite \eqref{e:boundarytermn} as
\begin{equation}
\partial_\upsilon \phi=-\phi^{-1}|\nabla u|^2(H_{P_1}-\kappa_{\partial\Sigma_t}).
\end{equation}
In particular, applying the coarea formula for the restriction $u|_{P_1}$--and using the homogeneous Neumann condition to see that $|\nabla u|_{P_1}|=|\nabla u|$ along $P_1$--we find that
\begin{equation}
\int_{P_1}\partial_{\upsilon}\phi dA=-\int_{t\in \mathcal{B}}\int_{\partial \Sigma_t \cap P_1}\phi^{-1}|\nabla u|(H_{P_1}-\kappa_{\partial\Sigma_t})+\int_{t\in \mathcal{A}}\int_{\partial \Sigma_t\cap P_1}|\nabla u|^{-1}\phi^{-1}\langle \nabla_{\nabla u}\nabla u,\upsilon\rangle.
\end{equation}
Since
\begin{equation}
|\nabla u|^{-1}\phi^{-1}|\langle \nabla_{\nabla u}\nabla u,\upsilon\rangle| \leq |II_{P_1}|\leq C,
\end{equation}
it follows that
\begin{equation}\label{p1_comp}
\int_{P_1}\partial_{\upsilon}\phi dA\leq -\int_{t\in \mathcal{B}}\int_{\partial \Sigma_t \cap P_1}\phi^{-1}|\nabla u|(H_{P_1}-\kappa_{\partial\Sigma_t})+C\int_{t\in \mathcal{A}}\mathcal{H}^1(\partial \Sigma_t\cap P_1).
\end{equation}

Apply \eqref{p1_comp} in \eqref{p0o9} to obtain
\begin{align}
\frac{1}{2}\int_{t\in\mathcal{B}}\int_{\Sigma_t}\frac{|\nabla u|}{\phi}\left(\frac{|\nabla^2 u|^2}{|\nabla u|^2}+R_g\right)dAdt
\leq &\int_{t\in \mathcal{B}}\left(\frac{1}{2}\int_{\Sigma_t}\frac{|\nabla u|}{\phi}R_{\Sigma_t}dA+\int_{\partial\Sigma_t\cap P_1}\frac{|\nabla u|}{\phi}(\kappa_{\partial \Sigma_t}-H_{P_1})\right)dt\notag\\
&+\int_{P_2}\partial_{\upsilon}\phi dA+C\int_{t\in A}\left(\mathcal{H}^2(\Sigma_t)+\mathcal{H}^1(\partial\Sigma_t \cap P_1)\right).
\end{align}
Observe that $|\nabla u|$ is uniformly bounded from below on $u^{-1}(\mathcal{B})$, since $\mathcal{B}$ is a closed subset of the regular values of $u$. Recalling that $\phi=(|\nabla u|^2+\varepsilon)^{1/2}$, we may take $\varepsilon\to 0$ in the preceding inequality to conclude that
\begin{align}
\begin{split}
\frac{1}{2}\int_{t\in\mathcal{B}}\int_{\Sigma_t}\left(\frac{|\nabla^2 u|^2}{|\nabla u|^2}+R_g\right)dAdt
\leq &\int_{t\in \mathcal{B}}\left(\frac{1}{2}\int_{\Sigma_t}R_{\Sigma_t}dA+\int_{\partial\Sigma_t\cap P_1}(\kappa_{\partial \Sigma_t}-H_{P_1})\right)dt\\
&+\int_{P_2}\partial_{\upsilon}|\nabla u|dA+C\int_{t\in A}\left(\mathcal{H}^2(\Sigma_t)+\mathcal{H}^1(\partial\Sigma_t \cap P_1)\right)\\
=&\int_{t\in \mathcal{B}}\left(2\pi \chi(\Sigma_t)-\int_{\partial \Sigma_t\cap P_2}\kappa_{\partial \Sigma_t}-\int_{\partial \Sigma_t\cap P_1}H_{P_1}\right)dt\\
&+\int_{P_2}\partial_{\upsilon}|\nabla u|dA+C\int_{t\in A}\left(\mathcal{H}^2(\Sigma_t)+\mathcal{H}^1(\partial\Sigma_t \cap P_1)\right),
\end{split}
\end{align}
where in the second step we have applied the Gauss-Bonnet theorem to $\Sigma_t$.

Finally, by Sard's theorem, we may take the measure $|\mathcal{A}|$ of $\mathcal{A}$ to be arbitrarily small. Since
\begin{equation}
t\mapsto \mathcal{H}^2(\Sigma_t)+\mathcal{H}^1(\Sigma_t\cap P_1)
\end{equation}
is integrable over $[\underline{u},\overline{u}]$ by the coarea formula, taking $|\mathcal{A}|\to 0$ in the preceding inequality yields the desired conclusion.
\end{proof}

\begin{remark} Note that for our applications in Section \ref{sec6}, the boundary component $P_2$ in Proposition \ref{prop:gen.stern} will be a piecewise smooth surface, diffeomorphic to the boundary $\partial (D^2\times [0,1])$ of the solid cylinder $D^2\times [0,1]$. However, our harmonic function $u$ in this case will be constant on the disks $D^2\times \{0\}$ and $D^2\times \{1\}$, with nonvanishing gradient along the cylindrical portion $S^1\times [0,1]$, and it is easy to check that the preceding argument carries over to this case without difficulty.
\end{remark}

\section{The Schwarzschildian Approach}
\label{sec5} \setcounter{equation}{0}
\setcounter{section}{5}

In this section we prove Theorem \ref{pmt} by establishing Theorem \ref{scpmt} and applying the Schwarzschildian reduction of Proposition \ref{simplify}. Unless stated otherwise, in this section $(M,g)$ will denote a 3-dimensional Schwarzschildian manifold.

\subsection{Connectivity of level sets}
Consider a coordinate sphere $S_r \subset M_{end}$ and let $M_r$ be the compact component of $M\setminus S_r$. In order to apply the identity \eqref{bdry.strn.1} to $\Omega=M_r$, a computation of the Gauss curvature piece is required, which is given in Proposition \ref{geodesiccurvature} below. Before proceeding to this calculation,  properties concerning the topology of regular level sets in $M_r$ will be recorded. Let $u$ be an asymptotically linear harmonic function as in Lemma \ref{lharmonic}, and for $t\in\mathbb{R}$, $r>0$ set $\Sigma_t^r=\Sigma_t\cap M_r$.

\begin{lemma}\label{l:connected}
Let $(M,g)$ and $u$ be as above. There are constants $r_0>0$ and $c_0>0$ such that for all $r\geq r_0$ and $t\in[-r+c_0,r-c_0]$ with $t$ a regular value, $\Sigma_t^r$ is connected with boundary $\partial\Sigma_t^r= S^1$.
\end{lemma}

\begin{proof}
Let $\ell=a_i x^i$ be the nontrivial linear function on $M_{end}$ to which $u$ converges.
By performing an orthogonal transformation if necessary, noting that this does not disturb the Schwarzschild asymptotics of $g$, it may be assumed that $\ell=a_1 x^1$ for some $a_1\neq 0$. Since $a_1^{-1} u$ is harmonic and has the same level sets as $u$, we may assume without loss of generality that $\ell=x^1$. In what follows $x^1$ will be denoted by $x$ for convenience.

The first step is to show that $\Sigma_t$ transversely intersects $S_r\subset M_{end}$, away from the north and south pole on the $x$-axis. More precisely,
we claim that there exist $r_0>0$ and $c_0>0$ such that for $r\geq r_0$ and $|t|\leq r-c_0$, $\Sigma_t$ transversely intersects $S_r$. To see this observe that using
\eqref{ux}, \eqref{uy}, and \eqref{uz} yields
\begin{equation}\label{e:angle1}
\delta(\nabla u,\partial_r)
=\frac{x}{r}\left(1+\frac{m}{2r}\right)^{-1}
+\frac{mx}{2r^2}\left(1+\frac{m}{2r}\right)^{-2}+O(r^{-2})
=\frac{x}{r}+O(r^{-2}),
\end{equation}
and
\begin{equation}
|\nabla u|_{\delta}=1-\frac{m}{2r}+\frac{mx^2}{2r^3}+O(r^{-2}),\quad\quad\quad
|\partial_r|_{\delta}=1.
\end{equation}
Therefore
\begin{equation}
\frac{\delta(\nabla u,\partial_r)}{|\nabla u|_{\delta}|\partial_r|_{\delta}}
=\frac{x}{r}\left(1+\frac{m}{2r}-\frac{mx^2}{2r^3}\right)+O(r^{-2}),
\end{equation}
so that for $|x|\leq r-c_{*}$ with appropriately chosen $c_* >0$ we have
\begin{equation}\label{nnnn}
|\cos\theta|=\frac{|\delta(\nabla u,\partial_r)|}{|\nabla u|_{\delta}|\partial_r|_{\delta}}\leq
1-\frac{1}{r}+O(r^{-2}).
\end{equation}
Here $\theta$ represents the angle between $\nabla u$ and $\partial_r$, which stays away from zero for large $r$. The desired claim now follows since $t=x+O(1)$ on $\Sigma_t \cap M_{end}$.

Now let $r\geq r_0$ and $|t|\leq r-c_0$ so that $\Sigma_t$ intersects $S_r$ transversely. Additionally, suppose that $t$ is a regular value of $u$. Since $\Sigma_t\cap S_r$ is transverse and nonempty, it consists of a finite number of disjoint embedded circles $\gamma_1,\dots,\gamma_p$. Since, by Lemma \ref{lharmonic}, $\Sigma_t$ is connected and noncompact with only one end, removing the circles yields the decomposition
\begin{equation}
\Sigma_t\setminus\cup_{i=1}^p\gamma_i=U\sqcup \mathcal{C},
\end{equation}
where $U$ is unbounded and connected, and $\overline{\mathcal{C}}$ is bounded and compact. Evidently $\Sigma_t^r\subset \overline{\mathcal{C}}$. If $\overline{\mathcal{C}}\neq \Sigma_t^r$, then there is a path component $\mathcal{C}'\subset \overline{\mathcal{C}}$ which lies outside of $M_r$. Since $\mathcal{C}'$ is compact, there is a largest $r'>r$ so that $S_{r'}\cap \mathcal{C}'\neq\emptyset$, see Figure \ref{pic:lemma5-1}. In this intersection, $S_{r'}$ is tangential to $\Sigma_t$. This, however, contradicts the transversality established above.
We conclude that $\Sigma_t^r=\mathcal{C}$.

\begin{figure}[hbt!]
\begin{picture}(0,0)
\put(200,25){\Large{$\Sigma_t^r$}}
\put(310,17){\color{red}{$S_{r}$}}
\put(330,55){\color{blue}{$S_{r'}$}}
\put(350,107){$\mathcal{C}'$}
\put(345,110){\vector(-4,-1){55}}
\end{picture}
\includegraphics[scale=.8]{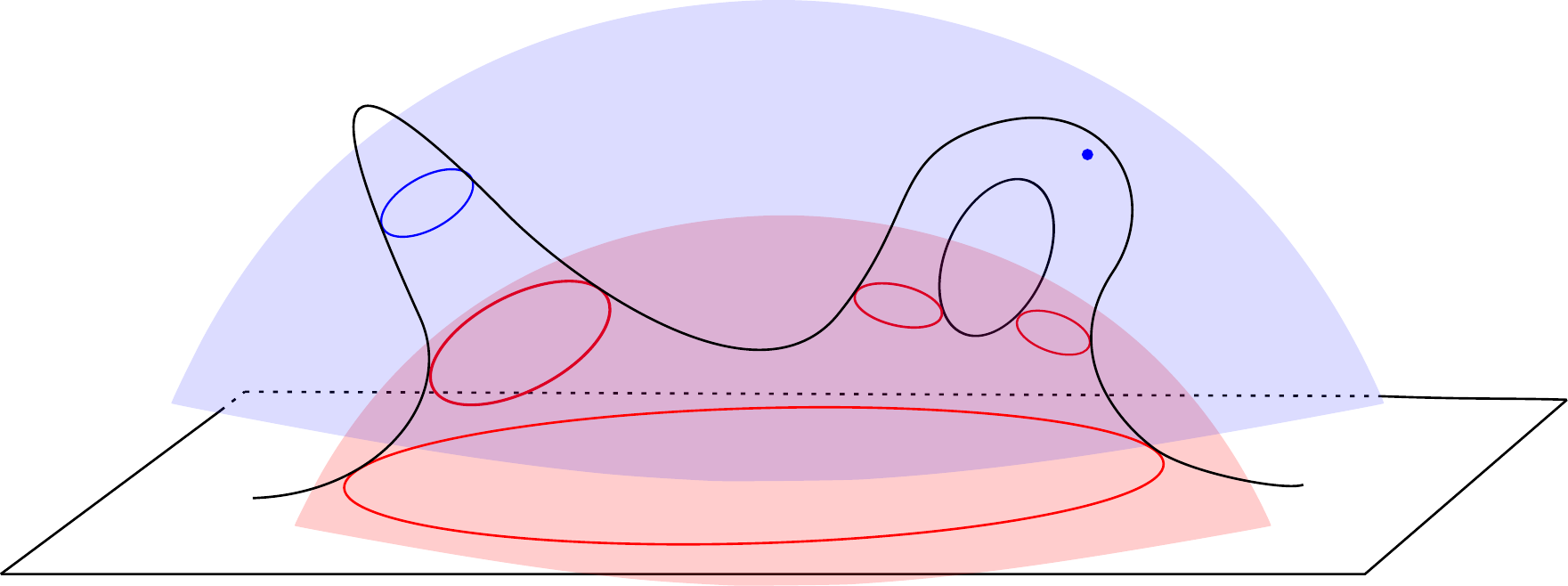}
\caption{The case where $\overline{\mathcal{C}}\neq\Sigma^r_t$ in the proof of Lemma \ref{l:connected}.}\label{pic:lemma5-1}
\end{figure}

Now assume that $\partial\Sigma_t^r$ is disconnected, so that there are at least two distinct circles $\gamma_1$ and $\gamma_2$. Since $U$ is connected, there is a path in $\overline{U}$ from $\gamma_1$ to $\gamma_2$. Let $r'>r$ be the smallest radius so that there is such a path $\sigma$ from $\gamma_1$ to $\gamma_2$ which lies entirely within $\Sigma_t^{r'}$, see Figure \ref{pic:lemma5-1-1}. Since $r'$ is the smallest such radius, $\sigma$ must intersect $S_{r'}$. At the intersection $\Sigma_t$ will be tangential to $S_{r'}$ since all perturbations of $\sigma$, within $\Sigma_t$ and supported on this intersection, either push $\sigma$ outside of $\Sigma_t^{r'}$ or possess a point of tangency. This again contradicts transversality and we conclude that $\partial\Sigma_t^r$ is connected. Since $\Sigma_t^r$ has no closed components, all points in $\Sigma_t^r$ can be connected to its boundary and we conclude that $\Sigma_t^r$ itself is connected.
\end{proof}

\begin{figure}[hbt!]
\begin{picture}(0,0)
\put(180,47){\large{$\Sigma_t^r$}}
\put(90,70){$\gamma_1$}
\put(98,65){\vector(1,-2){13}}
\put(120,100){$\gamma_2$}
\put(134,99){\vector(3,-1){38}}
\put(318,28){$\sigma$}
\put(115,13){{\color{red}{$S_r$}}}
\put(80,47){{\color{blue}{$S_{r'}$}}}
\end{picture}
\includegraphics[scale=.6]{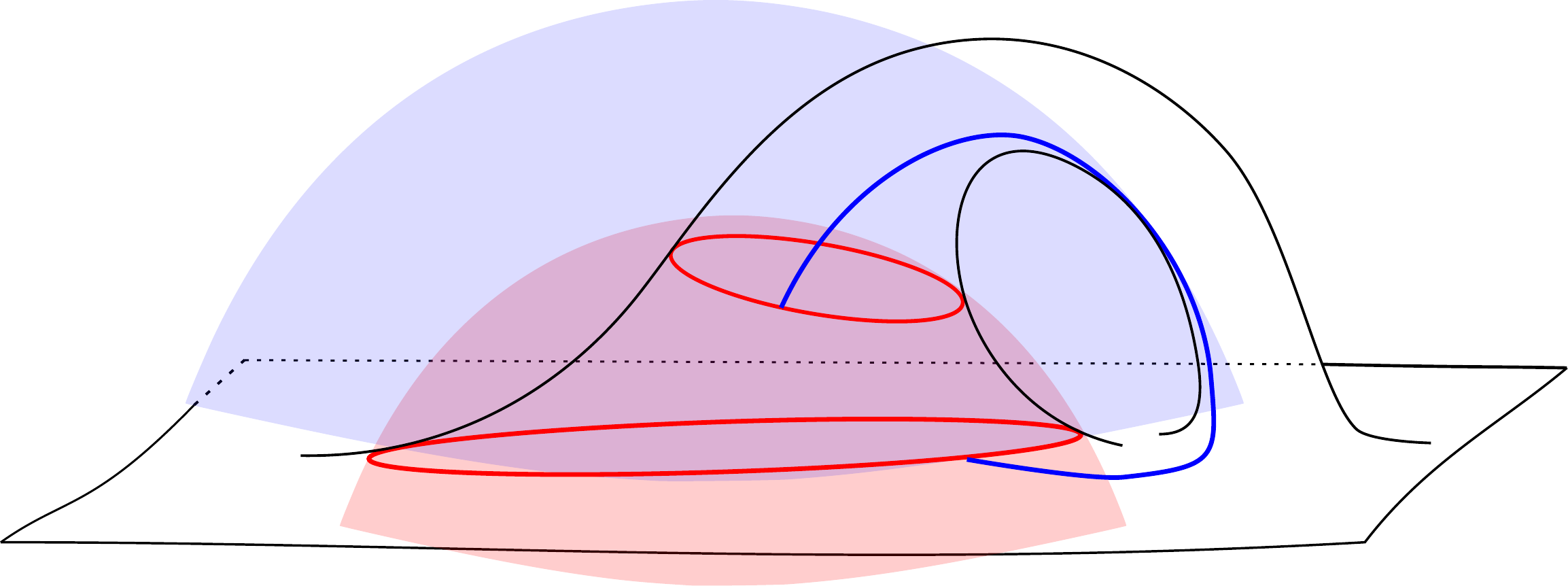}
\caption{The argument showing $\partial\Sigma_t^r$ is connected in Lemma \ref{l:connected}.}\label{pic:lemma5-1-1}
\end{figure}
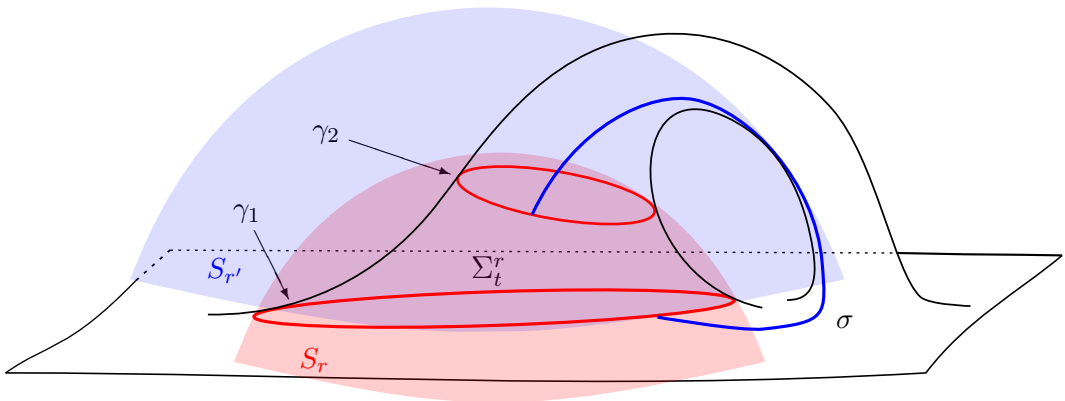

\subsection{The Gaussian curvature}
The Gaussian curvature integral appearing in formula \eqref{bdry.strn.1} in Proposition \ref{prop:gen.stern} will now be computed.

\begin{prop}\label{geodesiccurvature}
Suppose that $(M,g)$ is complete and Schwarzschildian.
Let $u$ be a harmonic function of Lemma \ref{lharmonic} which is asymptotic to a linear function $\ell$, and let $\overline{u}$ and $\underline{u}$ denote the maximum and minimum values of $u$ within $M_r$, respectively. Then
\begin{equation}\label{ty67}
\int_{\underline{u}}^{\overline{u}}
\int_{\Sigma_t^r}K dA dt\leq m\omega +O(r^{-1})
\end{equation}
for some constant $\omega\in(0,\infty)$ independent of $r$, where $K$ is the Gaussian curvature of $\Sigma_t^r$.
\end{prop}

\begin{proof}
As in the proof of Lemma \ref{l:connected} we may assume without loss of generality that $\ell=x$, where the asymptotic coordinates on $M_{end}$ will be denoted by $(x,y,z)$. Observe that on a $t$-level set
\begin{equation}
t=u=\frac{x}{1+\frac{m}{2r}}+\frac{a}{r}+O_2(r^{-2}),
\end{equation}
which implies that
\begin{equation}\label{x}
x=t+\frac{c(t)}{r}+O_2(r^{-2}),\quad\quad\quad c(t)=\frac{tm}{2}-a.
\end{equation}
In the expressions to follow, the subindex $l$ of $O_l$ will be ignored for convenience.
By the implicit function theorem we may solve for $x=x(y,z)$ when $r$ is large. Let
\begin{equation}
r^2=x^2+y^2+z^2=x^2+\rho^2,\quad\quad\quad \tilde{r}^2=t^2 +\rho^2,
\end{equation}
then a calculation shows that
\begin{equation}
x(y,z)
=t+\frac{c(t)}{\tilde{r}}+O(\tilde{r}^{-1}).
\end{equation}
Furthermore, in the asymptotic end
\begin{equation}\label{gmetricexp}
g=\left(1+\frac{m}{2r}\right)^4\left(dx^2
+dy^2 +dz^2\right),
\end{equation}
so that the induced metric on $\Sigma_{t}^r \cap M_{end}$ is given by
\begin{equation}
\gamma=\left(1+\frac{m}{2r}\right)^4\left((1+x_y^2)dy^2 +2x_y x_z dy dz
+(1+x_z^2) dz^2\right).
\end{equation}
From \eqref{x} the partial derivatives may be computed
\begin{equation}
x_y=-\frac{c}{r^3}(x x_y +y)+O\left(\frac{|x x_y|+|y|}{r^4}+\frac{1}{r^3}\right)
\quad\quad\Rightarrow\quad\quad x_y =-\frac{c(t) y}{r^3}+O(r^{-2}),
\end{equation}
and similarly
\begin{equation}
x_z=-\frac{c(t) z}{r^3}+O(r^{-2}).
\end{equation}
Hence
\begin{equation}\label{e:gamma}
\gamma=\left(1+\frac{m}{2r}\right)^4\left[\left(1+\frac{c^2 y^2}{r^6}\right) dy^2 +\frac{2c^2 y z}{r^6} dy dz
+\left(1+\frac{c^2 z^2}{r^6}\right) dz^2 +O(r^{-3})dx^i dx^j\right].
\end{equation}
Since the Gauss curvature consists of second derivatives and quadratic first derivatives, it follows that
\begin{equation}
K=O(r^{-3}).
\end{equation}

Let $r_0>0$ and $c_0>0$ be the constants given by Lemma \ref{l:connected}. From now on, we will only consider $r\geq r_0$. Let $\overline{u}$ and $\underline{u}$ be the max and min levels for $u$ within $M_r$. Then $\overline{u}=r-\frac{m}{2}+O(r^{-1})$ and $\underline{u}=-r-\frac{m}{2}+O(r^{-1})$. At this point, we will break the interval $[\underline{u},\overline{u}]$ into three pieces: $[\underline{u},-r+c_0]$, $[-r+c_0,r-c_0]$, and $[r-c_0,\overline{u}]$, see Figure \ref{pic:lemma5-2}. Consider $t\in[r-c_0,\overline{u}]$, then $0\leq\rho\leq c_1\sqrt{r}$ for some constant $c_1$. On this region all $t$-levels are regular and
\begin{equation}
\int_{\Sigma_t^r}K dA=O\left(\frac{\rho^2}{r^3}\right)=O(r^{-2}),
\end{equation}
so that
\begin{equation}\label{gauss1}
\int_{r-c_0}^{\overline{u}}\left(\int_{\Sigma_t^r}K dA\right) dt
=O(r^{-1}),
\end{equation}
and similarly
\begin{equation}\label{gauss2}
\int_{\underline{u}}^{-r+c_0}\left(\int_{\Sigma_t^r}K dA\right) dt=O(r^{-1}).
\end{equation}

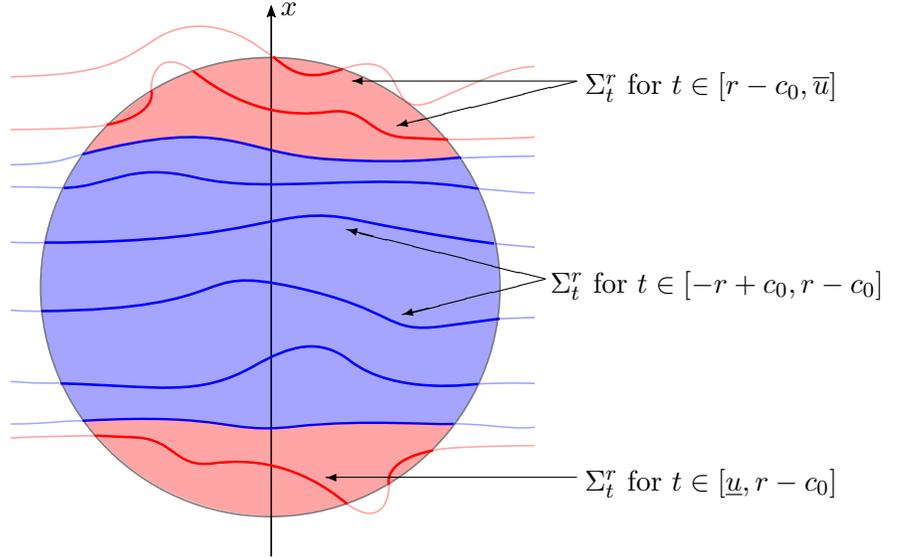
\begin{figure}[hbt!]
\includegraphics[scale=.8]{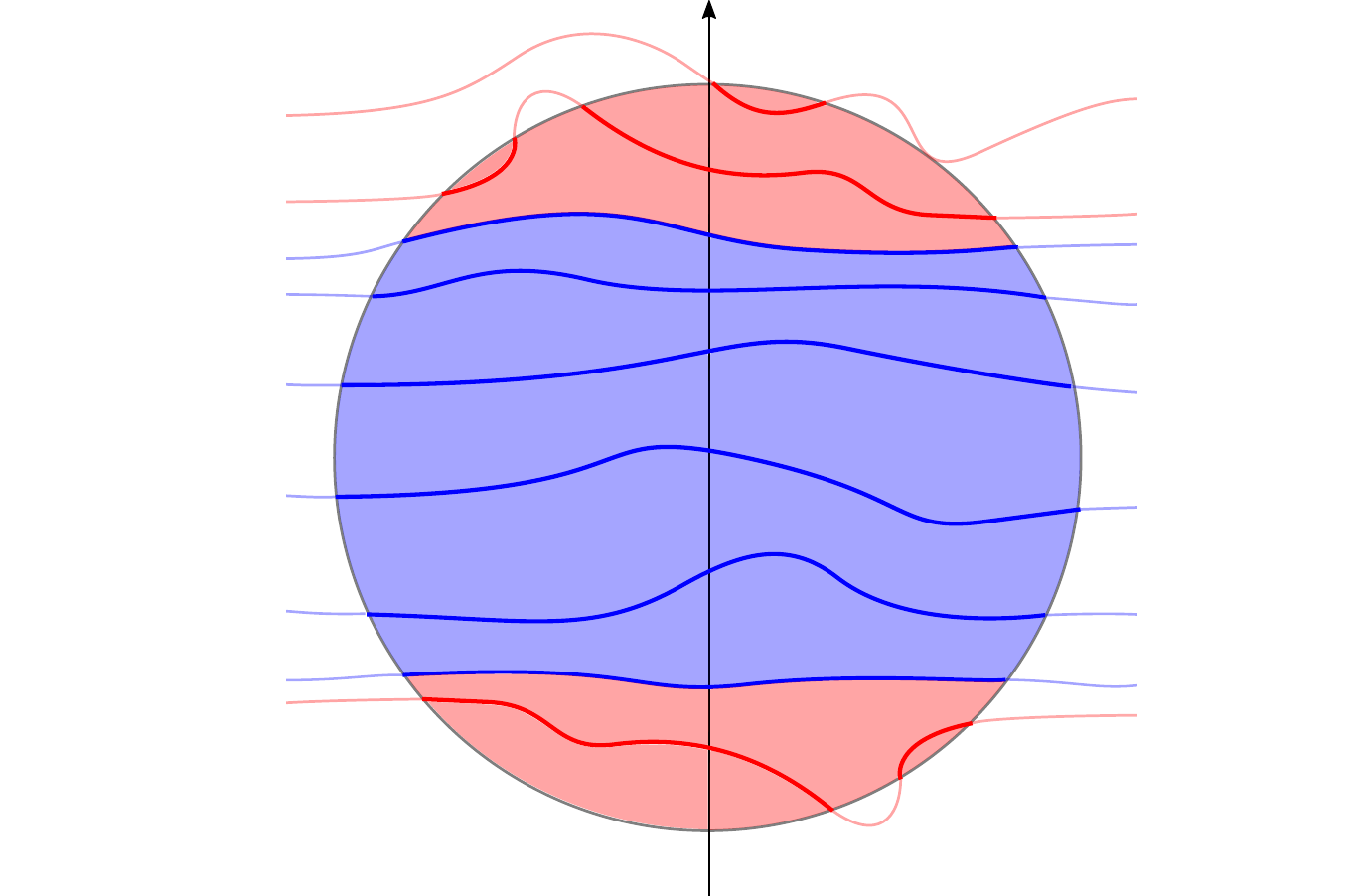}
\begin{picture}(0,0)
\put(-150,205){$x$}
\put(-35,175){$\Sigma_t^r$ for $t\in[r-c_0,\overline{u}]$}
\put(-38,180){\vector(-1,0){85}}
\put(-38,180){\vector(-4,-1){68}}
\put(-48,100){$\Sigma_t^r$ for $t\in[-r+c_0,r-c_0]$}
\put(-50,105){\vector(-4,1){75}}
\put(-50,105){\vector(-4,-1){54}}
\put(-35,25){$\Sigma_t^r$ for $t\in[\underline{u},r-c_0]$}
\put(-38,30){\vector(-1,0){95}}
\end{picture}
\caption{The decomposition of $M_r$ used to estimate the integral in Proposition \ref{geodesiccurvature}.}\label{pic:lemma5-2}
\end{figure}

Now let us restrict attention to the range $|t|\leq r-c_0$, so that $c_2\sqrt{r}\leq \rho\leq r$ for some constant $c_2>0$. According to Lemma \ref{l:connected}, for regular values $t$ in this range, $\Sigma_t^r$ is a connected smooth submanifold with boundary $\partial\Sigma_t^r =S^1 \subset S_r$. Let $\alpha:[0,\theta_0]\to\partial\Sigma_t^r$ be a parameterization and let $\tilde{\nu}$ be an inward pointing normal vector to $\partial\Sigma_t^r$ tangent to $\Sigma_t^r$, both to be chosen later. The geodesic curvature of the circle $\partial\Sigma_t^r$ is given by
\begin{equation}\label{geodesickurv}
\kappa=\left\langle\nu,\nabla_{\frac{\alpha'}{|\alpha'|}}\frac{\alpha'}{|\alpha'|}
\right\rangle
=\frac{\langle\tilde{\nu},\nabla_{\alpha'}\alpha'\rangle}{|\tilde{\nu}||\alpha'|^2},
\end{equation}
where $\nabla$ is the Levi-Civita connection for $g=\langle\cdot,\cdot\rangle$, $\nu=\tfrac{\tilde{\nu}}{|\tilde{\nu}|}$ is the unit normal of $\partial\Sigma_t^r$ tangent to $\Sigma_t^r$, and $\alpha'=\partial_\theta\alpha$ is the velocity vector associated with $\alpha$. Write this curve as
\begin{equation}
\alpha(\theta)=\left(x(\theta),y(\theta),z(\theta)\right)
\end{equation}
where these functions are defined by the equations
\begin{equation}
x(\theta)^2+y(\theta)^2+z(\theta)^2=r^2,\quad\quad\quad u(\alpha(\theta))=t.
\end{equation}
In order to compute $\alpha'$, observe that the equations defining $\alpha'$ (up to scaling) are
\begin{equation}
\alpha \cdot \alpha'=0,\quad\quad\quad \nabla u\cdot\alpha'=0,
\end{equation}
where $\cdot$ represents the Euclidean inner product. It follows that $\alpha$ and $\theta_0$ can be chosen so that
\begin{equation}
\alpha'=(zu_y -y u_z)\partial_x +(x u_z-z u_x)\partial_y +(yu_x -x u_y)\partial_z.
\end{equation}
The partial derivatives have the expansions
\begin{equation}\label{ux}
u_x=\left(1+\frac{m}{2r}\right)^{-1}
-x\left(1+\frac{m}{2r}\right)^{-2}\left(\frac{-mx}{2r^3}\right)
-\frac{ax}{r^3}+O(r^{-2})=1-\frac{m}{2r}+\frac{mx^2}{2r^3}+O(r^{-2}),
\end{equation}
\begin{equation}\label{uy}
u_y=-x\left(1+\frac{m}{2r}\right)^{-2}\left(\frac{-my}{2r^3}\right)
-\frac{ay}{r^3}+O(r^{-2})
=\frac{mxy}{2r^3}+O\left(\frac{\rho}{r^3}+\frac{1}{r^2}\right),
\end{equation}
\begin{equation}\label{uz}
u_z=-x\left(1+\frac{m}{2r}\right)^{-2}\left(\frac{-mz}{2r^3}\right)-\frac{az}{r^3}+O(r^{-2})
=\frac{mxz}{2r^3}+O\left(\frac{\rho}{r^3}+\frac{1}{r^2}\right).
\end{equation}
Therefore
\begin{equation}
\alpha'=\alpha'_x \partial_x+\alpha'_y\partial_y
+\alpha'_z \partial_z=O\left(\frac{\rho}{r^2}\right)\partial_x
+\left(-z+\frac{mz}{2r}+O\left(\frac{\rho}{r^2}\right)\right)\partial_y
+\left(y-\frac{my}{2r}+O\left(\frac{\rho}{r^2}\right)\right)\partial_z,
\end{equation}
and
\begin{equation}\label{e:speed}
|\alpha'|^2=\left(1+\frac{m}{2r}\right)^4 \left[O\left(\frac{\rho^2}{r^2}\right)
+\rho^2\left(1-\frac{m}{r}+\frac{m^2}{4r^2}\right)\right]
=\rho^2\left(1+\frac{m}{r}+O(r^{-2})\right).
\end{equation}

At this point, it is convenient to estimate the value $\theta_0$ of the parameterizing interval. On one hand, the length of $\partial\Sigma_t^r$ may be computed from (\ref{e:speed}),
\begin{equation}\label{e:length1}
\mathrm{Length}(\partial\Sigma_t^r)=\int_0^{\theta_0}|\alpha'|d\theta
=\int_0^{\theta_0}\rho\left(1+\frac{m}{2r}+O(r^{-2})\right)d\theta.
\end{equation}
On the other hand, we can parameterize the $yz$-projection of $\partial\Sigma_t^r$ by $\vartheta\mapsto(\rho(\vartheta)\cos\vartheta,\rho(\vartheta)\sin\vartheta)$ for $\vartheta\in[0,2\pi]$, and use (\ref{e:gamma}) to find
\begin{equation}\label{e:length2}
\mathrm{Length}(\partial\Sigma_t^r)
=\int_0^{2\pi}\sqrt{\det{\gamma|_{\partial\Sigma_t^r}}}d\vartheta
=\int_0^{2\pi}\rho\left(1+\frac{m}{r}+O(r^{-2})\right)d\vartheta.
\end{equation}
Using (\ref{x}), it follows that $\rho=\sqrt{r^2-x^2}$ is a constant (depending on $r$ and $t$) along $\partial\Sigma_t^r$ up to $O(r^{-2})$. Thus we may subtract (\ref{e:length1}) and (\ref{e:length2}) to obtain
\begin{equation}\label{e:length3}
\theta_0=2\pi\left(1+\frac{m}{2r}+O(r^{-2})\right).
\end{equation}

Let us return to our calculation of (\ref{geodesickurv}). The normal vector $\tilde{\nu}$ must satisfy
\begin{equation}
\alpha' \cdot\tilde{\nu}=0,\quad\quad\quad \nabla u\cdot\tilde{\nu}=0,
\end{equation}
and so we may choose
\begin{equation}
\tilde{\nu}=(\alpha'_z u_y-\alpha'_y u_z)\partial_x
+(\alpha'_x u_z-\alpha'_z u_x)\partial_y
+(\alpha'_y u_x-\alpha'_x u_y)\partial_z.
\end{equation}
It follows that the components have the expansions
\begin{equation}
\tilde{\nu}_x =\alpha'_z u_y-\alpha'_y u_z
=\frac{mx\rho^2}{2r^3}+O\left(\frac{\rho}{r^2}\right),
\end{equation}
\begin{equation}
\tilde{\nu}_y =\alpha'_x u_z-\alpha'_z u_x
=-y+\frac{my}{r}-\frac{mx^2 y}{2r^3}+O\left(\frac{\rho}{r^2}\right),
\end{equation}
\begin{equation}
\tilde{\nu}_z=\alpha'_y u_x-\alpha'_x u_y=-z+\frac{mz}{r}-\frac{mx^2 z}{2r^3}+O\left(\frac{\rho}{r^2}\right),
\end{equation}
and
\begin{equation}
|\tilde{\nu}|^2=\rho^2\left(1+\frac{m}{2r}\right)^4 \left(1-\frac{2m}{r}+\frac{mx^2}{r^3}+O(r^{-2})\right)
=\rho^2\left(1+\frac{mx^2}{r^3}+O(r^{-2})\right).
\end{equation}

We now compute the covariant derivative portion of \eqref{geodesickurv}. Observe that
\begin{equation}
\nabla_{\alpha'}\alpha'=\alpha'^i \nabla_i\left(\alpha'^j\partial_j\right)
=\left(\alpha'^i \partial_i \alpha'^j \right)\partial_j +\alpha'^i \alpha'^j \Gamma_{ij}^l \partial_l,
\end{equation}
where $\Gamma_{ij}^l$ are Christoffel symbols. Furthermore
\begin{equation}
\alpha'^i \partial_i \alpha'_x =O\left(\frac{\rho}{r^2}\right),
\end{equation}
\begin{align}
\begin{split}
\alpha'^i \partial_i \alpha'_y=&
O\left(\frac{\rho}{r^2}\right)O\left(\frac{\rho}{r^2}\right)
+\left(-z+\frac{mz}{2r}+O\left(\frac{\rho}{r^2}\right)\right)
\left(-\frac{myz}{2r^3}+O\left(\frac{1}{r^2}+\frac{\rho^2}{r^4}\right)\right)\\
&
+\left(y-\frac{my}{2r}+O\left(\frac{\rho}{r^2}\right)\right)
\left(-1+\frac{m}{2r}-\frac{mz^2}{2r^3}+O\left(\frac{1}{r^2}+\frac{\rho^2}{r^4}\right)\right)\\
=&-y+\frac{my}{r}+O\left(\frac{\rho}{r^2}+\frac{\rho^3}{r^4}\right),
\end{split}
\end{align}
\begin{align}
\begin{split}
\alpha'^i \partial_i \alpha'_z=&
O\left(\frac{\rho}{r^2}\right)O\left(\frac{\rho}{r^2}\right)
+\left(-z+\frac{mz}{2r}+O\left(\frac{\rho}{r^2}\right)\right)
\left(1-\frac{m}{2r}+\frac{my^2}{2r^3}+O\left(\frac{1}{r^2}+\frac{\rho^2}{r^4}\right)\right)\\
&
+\left(y-\frac{my}{2r}+O\left(\frac{\rho}{r^2}\right)\right)
\left(\frac{myz}{2r^3}+O\left(\frac{1}{r^2}+\frac{\rho^2}{r^4}\right)\right)
\\
=&-z+\frac{mz}{r}+O\left(\frac{\rho}{r^2}+\frac{\rho^3}{r^4}\right).
\end{split}
\end{align}
Hence
\begin{equation}
\nabla_{\alpha'}\alpha'=\left(-y+\frac{my}{r}\right)\partial_y
+\left(-z+\frac{mz}{r}\right)\partial_z +O\left(\frac{\rho}{r^2}\right) \partial_l
+\alpha'^i \alpha'^j \Gamma_{ij}^l \partial_l.
\end{equation}
To compute the Christoffel symbols write let $w=1+\tfrac{m}{2r}$ and use \eqref{gmetricexp} to find
\begin{align}
\begin{split}
\Gamma_{ij}^l=&\frac{1}{2}g^{lk}\left(\partial_i g_{kj}+\partial_j g_{ki}-\partial_k g_{ij}\right)\\
=& \frac{1}{2} w^{-4}\delta^{lk}\left(\delta_{kj} \partial_i w^4
+\delta_{ki}\partial_j w^4 -\delta_{ij} \partial_k w^4 \right)\\
=&2\left(\delta_{j}^l\partial_i\log w +\delta_{i}^l \partial_j \log w
-\delta_{ij} \partial^l \log w\right).
\end{split}
\end{align}
Therefore using the orthogonality of $\tilde{\nu}$ and $\alpha'$ yields
\begin{equation}
\langle \tilde{\nu},\alpha'^i \alpha'^j \Gamma_{ij}^l\partial_l\rangle =w^4 \delta_{kl}\tilde{\nu}^k \alpha'^i \alpha'^j \Gamma_{ij}^l=-2|\alpha'|^2 \tilde{\nu}^l \partial_l \log w.
\end{equation}
Since
\begin{equation}
\tilde{\nu}^l \partial_l \log w=\tilde{\nu}^l \left(1+\frac{m}{2r}\right)^{-1}
\left(-\frac{mx^l}{2r^3}\right)
=\frac{m\rho^2}{2r^3}+O(r^{-2}),
\end{equation}
we then have
\begin{equation}
\langle \tilde{\nu},\alpha'^i \alpha'^j \Gamma_{ij}^l\partial_l\rangle
=-\frac{m\rho^4}{r^3}+O\left(\frac{\rho^2}{r^2}\right).
\end{equation}

Putting this altogether produces
\begin{align}
\begin{split}
\langle\tilde{\nu},\nabla_{\alpha'}\alpha'\rangle
=&\left(1+\frac{m}{2r}\right)^4\left[O\left(\frac{\rho}{r^2}\right)
O\left(\frac{|x|\rho}{r^2}
+\frac{\rho}{r^2}\right)-\frac{m\rho^4}{r^3}
+O\left(\frac{\rho^2}{r^2}\right)\right.\\
&\left. +\left(-y+\frac{my}{r}-\frac{mx^2 y}{2r^3}+O\left(\frac{\rho}{r^2}\right)\right)
\left(-y+\frac{my}{r}+O\left(\frac{\rho}{r^2}\right)\right)\right.\\
&\left. +\left(-z+\frac{mz}{r}-\frac{mx^2 z}{2r^3}+O\left(\frac{\rho}{r^2}\right)\right)
\left(-z+\frac{mz}{r}+O\left(\frac{\rho}{r^2}\right)\right)\right]\\
=&\rho^2\left(1+\frac{mx^2}{2r^3}-\frac{m\rho^2}{r^3}+O(r^{-2})\right).
\end{split}
\end{align}
We also have
\begin{equation}
|\tilde{\nu}||\alpha'|
=\rho^2\left(1+\frac{m}{2r}+\frac{mx^2}{2r^3}+O(r^{-2})\right),
\end{equation}
and therefore
\begin{equation}
\frac{\langle\tilde{\nu},\nabla_{\alpha'}\alpha'\rangle}{|\tilde{\nu}||\alpha'|}
=1-\frac{m}{2r}-\frac{m\rho^2}{r^3}+O(r^{-2}).
\end{equation}
Combining this with (\ref{e:length3}) we find that
\begin{equation}
\int_{\partial\Sigma_t^r}\kappa ds=\int_{0}^{\theta_0}
\frac{\langle\tilde{\nu},\nabla_{\alpha'}\alpha'\rangle}{|\tilde{\nu}||\alpha'|}
d\theta
= 2\pi\left(1-\frac{m}{r}+\frac{mt^2}{r^3}\right)+O(r^{-2}).
\end{equation}

By Sard's theorem we may restrict attention to regular level sets when computing \eqref{ty67}. Moreover since for regular levels in the range $|t|\leq r-c_0$ with $r\geq r_0$, the compact surface $\Sigma_t^r$ is connected with nonempty boundary (Lemma \ref{l:connected}), its Euler
characteristic satisfies $\chi(\Sigma_t^r)\leq 1$. Thus using \eqref{gauss1}, \eqref{gauss2}, and the Gauss-Bonnet theorem we find that
\begin{align}
\begin{split}
\int_{\underline{u}}^{\overline{u}}
\int_{\Sigma_t^r}K d\mathcal{H} dt=&
\int_{-r+c_0}^{r-c_0}\left(2\pi\chi(\Sigma_t^r)-\int_{\Sigma_t^r}\kappa ds\right)dt\\
&+\int_{\underline{u}}^{-r+c_0}\left(\int_{\Sigma_{t}^r}KdA\right)dt
+\int_{r-c_0}^{\overline{u}}\left(\int_{\Sigma_{t}^r}KdA\right)dt\\
\leq& \frac{8\pi}{3}m+O(r^{-1}).
\end{split}
\end{align}
\end{proof}

\subsection{Proof of the positive mass theorem}
\begin{proof}[Proof of Theorem \ref{scpmt}]
Let $(M,g)$ be complete with nonnegative scalar curvature and Schwarzschildian. Consider the harmonic function $u$ of Lemma \ref{lharmonic} asymptotic to the linear function $\ell(x,y,z)=x$.
Apply identity \eqref{bdry.strn.1} of Proposition \ref{prop:gen.stern} to $\Omega=M_{r}$ to obtain
\begin{equation}\label{e:cylstern1}
\int_{S_r}\partial_{\upsilon}|\nabla u|dA
\geq\frac12\int_{\underline{u}}^{\overline{u}}
\int_{\Sigma_t^r}\left(\frac{|\nabla^2 u|^2}{|\nabla u|^2}+R_g-2K\right)dAdt,
\end{equation}
where $\overline{u}$ and $\underline{u}$ denote the maximum and minimum values of $u$ on $M_{r}$, and $K$ is the Gaussian curvature of $\Sigma_t^r$.

In order to compute the boundary integral in (\ref{e:cylstern1}), use \eqref{ux}, \eqref{uy}, and \eqref{uz} to find
\begin{equation}
|\nabla u|
=\left(1+\frac{m}{2r}\right)^{-2}\left(u_x^2 +u_y^2+u_z^2\right)^{\frac{1}{2}}
=1-\frac{3m}{2r}+\frac{mx^2}{2r^3}+O_1(r^{-2}).
\end{equation}
It follows that
\begin{equation}
\partial_\upsilon |\nabla u|=
\left(1+\frac{m}{2r}\right)^{-2}\partial_r |\nabla u|=\frac{3m}{2r^2}-\frac{m x^2}{2 r^4}+O(r^{-3}),
\end{equation}
and therefore
\begin{equation}
\int_{S_r}\partial_\upsilon |\nabla u|dA=4\pi\left(\frac{3m}{2}-\frac{m}{6}\right)+O(r^{-1})
=\frac{16\pi}{3}m+O(r^{-1}),
\end{equation}
where we have used
\begin{equation}
\int_{S_r}x^2 dA_{\delta}=\frac{1}{3}\int_{S_r}(x^2+y^2+z^2)dA_{\delta}
=\frac{1}{3}\int_{S_r} r^2 dA_{\delta}=\frac{4\pi}{3}r^4.
\end{equation}
This combined with Proposition \ref{geodesiccurvature} and letting $r\rightarrow\infty$ yields
\begin{equation}\label{hjksa}
16\pi m\geq
\int_{-\infty}^{\infty}\int_{\Sigma_t}\left(\frac{|\nabla^2 u|^2}{|\nabla u|^2}+R_g\right)dAdt,
\end{equation}
from which we find that $m\geq 0$.

Consider now the case of equality $m=0$. Inequality \eqref{hjksa} implies that $R_g\equiv 0$ and $|\nabla^2 u|\equiv 0$. In particular, $\nabla u$ is a parallel vector field. The same procedure above may be applied to second and third harmonic functions $v$ and $w$ of Lemma \ref{lharmonic} asymptotic to the linear functions $\ell=y$ and $\ell=z$, respectively, so that $\nabla v$ and $\nabla w$ are also parallel. Since these three vector fields are linearly independent, $(M,g)$ is flat. Since $(M,g)$ is also complete it must be isometric to Euclidean 3-space.
\end{proof}

\begin{proof}[Proof of Theorem \ref{pmt}]
Let $(M,g)$ be complete of nonnegative scalar curvature, and asymptotically flat with $m$ the mass of a designated end $M_{end}^+$. Let $(\overline{M},\overline{g})$ be the Schwarzschildian manifold of Proposition \ref{simplify} with mass $\overline{m}$ satisfying $|m-\overline{m}|<\varepsilon$.
According to Theorem \ref{scpmt}, $\overline{m}\geq 0$. Since $\varepsilon>0$ is arbitrarily small, we conclude that $m\geq 0$.

The conclusion in the case of equality, $m=0$, follows from the positive mass inequality as in \cite{SchoenYauI}. Namely, one shows through conformal deformation that $(M,g)$ is scalar flat, and then that it is Ricci flat via an infinitesimal Ricci flow.
\end{proof}

\section{The Harmonic Coordinate Method}
\label{sec6} \setcounter{equation}{0}
\setcounter{section}{6}

In this section, we give another way to derive the total mass of an asymptotically flat manifold. Instead of using the trick of approximating by Schwarzschild metrics as in the previous section, we show how the mass term falls out naturally from our boundary term at infinity. Let $(M,g)$ be a complete asymptotically flat Riemannian 3-manifold, and let $M_{ext}$ be the exterior region associated with a specified end $M_{end}$.
According to \cite[Lemma 4.1]{HuiskenIlmanen} the exterior region is diffeomorphic to $\mathbb{R}^3$ minus a finite number of disjoint balls, and has minimal boundary. Let $\{x^1,x^2,x^3\}$ be harmonic coordinates on $M_{ext}$ as in Section \ref{sec:harmoniccoords}, with homogeneous Neumann condition on $\partial M_{ext}$, and let $\vec{x}=(x^1,x^2,x^3)$. For a unit vector $a\in S^2\subset\mathbb{R}^3$, it obviously follows that $u=\vec{x}\cdot a$ is harmonic on $M_{ext}$ with homogeneous Neumann condition. For $L>0$ sufficiently large, consider the coordinate cylinders $C_L:=D_L^{\pm}\cup T_L$ where
\begin{equation}
D_L^{\pm}:=\{\vec{x}\mid \vec{x}\cdot a=\pm L,\text{ }|\vec{x}|^2-(\vec{x}\cdot a)^2\leq L^2\},\quad T_L:=\{\vec{x}\mid |\vec{x}\cdot a|\leq L,\text{ }|\vec{x}|^2-(\vec{x}\cdot a)^2=L^2\}.
\end{equation}
Set $\Omega_L\subset M_{ext}$ to be the closure of the bounded component of $M_{ext}\setminus C_{L}$.
Following the arguments of \cite[Section 4]{Bartnik}, if the scalar curvature $R_g$ is integrable then the mass of $M_{ext}$ is given by
\begin{equation}\label{mass.def}
m=\lim_{L\to\infty}\frac{1}{16\pi}\int_{C_L}\sum_{i}(g_{ij,i}-g_{ii,j})\upsilon^j dA.
\end{equation}
where $\upsilon$ is the outward unit normal to $C_L$.


\subsection{Computation of the mass}\label{sec:geocurv}

To prove inequality \eqref{masslowerb}, begin by applying Proposition \ref{prop:gen.stern} to $u$ on the cylindrical domains $\Omega_L$ (so that $P_2=C_L$ and $P_1=\partial M_{ext}$) to find that
\begin{align}
\frac{1}{2}\int_{\Omega_L}\left(\frac{|\nabla^2u|^2}{|\nabla u|}+R_g|\nabla u|\right)dV\leq  \int_{-L}^L\left(2\pi \chi(\Sigma_t^L)-\int_{\Sigma_t^L\cap T_L}\kappa_{t,L}\right)dt+\int_{C_L}\partial_{\upsilon}|\nabla u| dA,
\end{align}
where $\Sigma_t^L:=\{u=t\}\cap \Omega_L$, and $\kappa_{t,L}$ is the geodesic curvature of the curve $\Sigma_t^L\cap T_L$ viewed as the boundary of $\Sigma_t$. Note that the asymptotics guarantee that, for $L$ sufficiently large, the level sets $\Sigma_t^L$ indeed meet $T_L$ transversely.
We claim next that for every regular value $t\in (-L,L)$, $\Sigma_t^L$ consists of a single connected component, intersecting $T_L$ along the circle $\Sigma_t^L\cap T_L$. Indeed, if this is not the case, then there is a regular value $t\in (-L,L)$ and a component $\Sigma'\subset\Sigma_t^L$ disjoint from $T_L$. Since $M_{ext}$ is diffeomorphic to the compliment of finitely many balls in $\mathbb{R}^3$,
there is a domain $E\subset \Omega_L$ such that $\partial E\setminus \partial M_{ext}=\Sigma'$ and $E\cap T_L=\varnothing$. But since $u$ is harmonic with Neumann boundary conditions on $\partial M_{ext}$ and identically $t$ on $\Sigma'$, the maximum principle would then imply that $u\equiv t$ in $E$, contradicting the fact that $t$ is a regular value.
Thus, $\Sigma_t^L$ has only one component, with boundary given by $\Sigma_t^L\cap T_L$, and as a consequence $\chi(\Sigma_t^L)\leq 1.$
In particular, applying this in the preceding computation gives
\begin{align}\label{e:cptstern2}
\frac{1}{2}\int_{\Omega_L}\left(\frac{|\nabla^2 u|^2}{|\nabla u|}+R_g|\nabla u|\right)dV\leq  4\pi L-\int_{-L}^L\left(\int_{\Sigma_t^L\cap T_L}\kappa_{t,L}\right)dt+\int_{C_L}\partial_{\upsilon}|\nabla u|dA.
\end{align}

The remainder of the proof of Theorem \ref{mass.bd} rests on a computation of the boundary terms in inequality (\ref{e:cptstern2}). To carry out these computations, it will be useful to take $a=\partial_{x^1}$, so that $u=x^1$ is the distinguished coordinate. In what follows the notation $\int_{D_L^{\pm}}\pm f$ represents $\int_{D_L^{+}}f-\int_{D_L^{-}}f$.

\begin{lemma}\label{ddu.comp} In the notation fixed above, we have
\begin{align}
\begin{split}
\int_{C_L}\partial_{\upsilon}|\nabla u| dA=&\frac{1}{2}\int_{D_L^{\pm}}\pm\sum_{j}(g_{1j,j}-g_{jj,1})dA\\
&+\frac{1}{2L}\int_{T_L}\left[x^2(g_{21,1}-g_{11,2})+x^3(g_{31,1}-g_{11,3})\right]dA
+O(L^{1-2q}).
\end{split}
\end{align}
\end{lemma}

\begin{proof}
To begin, note that
\begin{equation}
\nabla |\nabla u|=\nabla(g^{11})^{1/2}=-\frac{1}{2}\nabla g_{11}+O(|x|^{-1-2q}),
\end{equation}
where in the second line we have used the decay rates \eqref{e:harmdecay}. Next since the outer normal $\upsilon$ to $C_L$ is given by
\begin{equation}
\upsilon=\pm \partial_1+O(|x|^{-q})\quad\text{ on }\quad D_L^{\pm},\quad\text{ and }\quad\upsilon=\frac{x^2\partial_2+x^3\partial_3}{L}+O(|x|^{-q})\quad\text{ on }T_L,
\end{equation}
it follows that
\begin{equation}
\int_{C_L}\partial_{\upsilon}|\nabla u| dA=-\frac{1}{2}\int_{D^\pm_L}\pm g_{11,1}dA-\frac{1}{2L}\int_{T_L}(x^2g_{11,2}+x^3g_{11,3})dA+O(L^{1-2q}).
\end{equation}
Now, because $x^1$ is harmonic we see that
\begin{align}
\begin{split}
g_{11,1}=&-2g(\nabla_{\partial_1}\partial_1,\partial_1)+O(|x|^{-1-2q})\\
=&2g(\nabla_{\partial_2}\partial_1,\partial_2)
+2g(\nabla_{\partial_3}\partial_1,\partial_3)+O(|x|^{-1-2q})\\
=&-2g_{21,2}-2g_{31,3}+g_{22,1}+g_{33,1}+O(|x|^{-1-2q}),
\end{split}
\end{align}
and therefore
\begin{align}
\begin{split}
\int_{C_L}\partial_{\upsilon}|\nabla u| dA=&\int_{D_L^{\pm}}\pm(g_{12,2}+g_{13,3}-\frac{1}{2}g_{22,1}-\frac{1}{2}g_{33,1})dA\\
&-\frac{1}{2L}\int_{T_L}(x^2g_{11,2}+x^3g_{11,3})dA+O(L^{1-2q})\\
=&\frac{1}{2}\int_{D_L^{\pm}}\pm(g_{12,2}-g_{22,1}+g_{13,3}-g_{33,1})dA\\
&+\int_{D_L^+}\frac{1}{2}(g_{12,2}+g_{13,3})dA-\int_{D_L^-}\frac{1}{2}(g_{12,2}+g_{13,3})dA\\
&-\frac{1}{2L}\int_{T_L}(x^2g_{11,2}+x^3g_{11,3})dA+O(L^{1-2q}).
\end{split}
\end{align}
Applying the divergence theorem to the penultimate line above, and subsequently employing the fundamental theorem of calculus on $T_L$ yields
\begin{align}
\begin{split}
\int_{C_L}\partial_{\upsilon}|\nabla u| dA=&\frac{1}{2}\int_{D_L^{\pm}}\pm(g_{12,2}-g_{22,1}+g_{13,3}-g_{33,1})dA\\
&+\int_{\partial D_L^+}\frac{1}{2L}(x^2g_{12}+x^3g_{13})dA-\int_{\partial D_L^-}\frac{1}{2L}(x^2g_{12}+x^3g_{13})dA\\
&-\frac{1}{2L}\int_{T_L}(x^2g_{11,2}+x^3g_{11,3})dA+O(L^{1-2q})\\
=&\frac{1}{2}\int_{D_L^{\pm}}\pm(g_{12,2}-g_{22,1}+g_{13,3}-g_{33,1})dA+\int_{T_L}\partial_1\left(\frac{x^2}{2L}g_{12}+\frac{x^3}{2L}g_{13}\right)dA\\
&-\frac{1}{2L}\int_{T_L}(x^2g_{11,2}+x^3g_{11,3})dA+O(L^{1-2q})\\
=&\frac{1}{2}\int_{D_L^{\pm}}\pm (g_{12,2}-g_{22,1}+g_{13,3}-g_{33,1})dA\\
&+\frac{1}{2L}\int_{T_L}[x^2(g_{21,1}-g_{11,2})+x^3(g_{31,1}-g_{11,3})]dA+O(L^{1-2q}).
\end{split}
\end{align}
\end{proof}

\begin{lemma}\label{curv.comp}
In the notation established above, we have
\begin{align}
\begin{split}
\int_{-L}^L\left(\int_{\Sigma_t^L\cap T_L}\kappa_{t,L}\right)dt=&4\pi L+\frac{1}{2L}\int_{T_L}\left[x^2(g_{33,2}-g_{23,3})+x^3(g_{22,3}-g_{32,2})\right]dA\\
&+O(L^{1-2q}+L^{-q}).
\end{split}
\end{align}
\end{lemma}

\begin{proof}
To begin, recall that the geodesic curvature $\kappa_{t,L}$ is given by
\begin{equation}
\kappa_{t,L}=\langle \nabla_{\tau}\beta,\tau\rangle=-\langle \beta,\nabla_{\tau}\tau\rangle,
\end{equation}
where $\tau$ is a unit tangent vector to $\Sigma_t^L\cap T_L$ and $\beta$ is the outward pointing unit normal to $\Sigma_t^L\cap T_L$ along $\Sigma_t^L$. Let
\begin{equation}
X:=x^2\partial_2+x^3\partial_3\quad\text{ and }\quad Y:=x^3\partial_2-x^2\partial_3.
\end{equation}
Then by setting $\tilde{X}:=X-\langle X,\tau\rangle \tau$ we may take
\begin{equation}
\tau=\frac{Y}{|Y|}\quad\text{ and }\quad\beta=\frac{\tilde{X}}{|\tilde{X}|}.
\end{equation}
Consequently
\begin{eqnarray}\label{e:geocurv2}
\kappa_{t,L}=-\left\langle \frac{\tilde{X}}{|\tilde{X}|},\nabla_{\tau}\tau\right\rangle
=\frac{-1}{|\tilde{X}||Y|^3}\left(|Y|\langle X,\nabla_YY\rangle-\langle X,Y\rangle\langle \nabla |Y|,Y\rangle\right).
\end{eqnarray}
The decay conditions \eqref{e:harmdecay} imply that
\begin{equation}
|\nabla|Y||=O(|x|^{-q})\quad\text{ and }\quad\langle X,Y\rangle=O(|x|^{2-q}).
\end{equation}
It follows that
\begin{equation}
\frac{\langle X,Y\rangle \langle \nabla|Y|,Y\rangle}{|\tilde{X}||Y|^3}
=O\left(\frac{|x|^{2-q}|x|^{-q}|x|}{|x|^4}\right)=O(|x|^{-1-2q}),
\end{equation}
and hence
\begin{equation}
\kappa_{t,L}=\frac{-\langle X,\nabla_YY\rangle}{|\tilde{X}||Y|^2}+O(|x|^{-1-2q}).
\end{equation}

A direct computation gives
\begin{equation}
\nabla_YY=-X+(x^3)^2\nabla_{\partial_2}\partial_2+(x^2)^2
\nabla_{\partial_3}\partial_3-2x^2x^3\nabla_{\partial_2}\partial_3.
\end{equation}
Upon expanding $\langle X,\nabla_YY\rangle=\langle x^2\partial_2+x^3\partial_3,\nabla_YY\rangle$ in terms of the metric derivatives, we see that \eqref{e:geocurv2} becomes
\begin{align}\label{e:geocurv3}
\begin{split}
\kappa_{t,L}=&\frac{|X|^2}{|\tilde{X}||Y|^2}+O(L^{-1-2q})-\frac{\langle X,(x^3)^2\nabla_{\partial_2}\partial_2+(x^2)^2\nabla_{\partial_3}
\partial_3-2x^2x^3\nabla_{\partial_2}\partial_3\rangle}{L^3}\\
=&\frac{|X|}{|Y|^2}+O(L^{-2-q}+L^{-1-2q})-\frac{1}{2L^3}(x^2(x^3)^2g_{22,2}
+(x^2)^2x^3g_{33,3})\\
&+[(x^2)^2x^3+\frac{1}{2}(x^3)^3]\frac{g_{22,3}}{L^3}
+[(x^3)^2x^2+\frac{1}{2}(x^2)^3]\frac{g_{33,2}}{L^3}
-(x^2)^3\frac{g_{23,3}}{L^3}-(x^3)^3\frac{g_{32,2}}{L^3},
\end{split}
\end{align}
where the decay properties \eqref{e:harmdecay} have been used repeatedly. At this point it will be useful to parameterize $\Sigma_t^L\cap T_L$ by
$$[0,2\pi]\ni s\mapsto \gamma(s):=(t, L\cos(s),L\sin(s)).$$
Notice that $\gamma'(s)=-Y$. We then have
\begin{align}\label{e:partgeocurv1}
\begin{split}
\int_{\Sigma_t^L\cap T_L}\frac{|X|}{|Y|^2}=&\int_0^{2\pi}
\left(|Y|\frac{|X|}{|Y|^2}\right)(\gamma(s))ds\\
=&\int_0^{2\pi}\left(1+\frac{|X|-|Y|}{|Y|}\right)(\gamma(s))ds\\
=&2\pi+\int_0^{2\pi}\left(\frac{|X|^2-|Y|^2}{|Y|(|X|+|Y|)}\right)
(\gamma(s))ds\\
=&2\pi+\frac{1}{2L^2}\int_0^{2\pi}(|X|^2-|Y|^2)(\gamma(s))ds+O(L^{-2q}).
\end{split}
\end{align}
Next compute
\begin{align}
\begin{split}
&\frac{1}{L^2}\int_0^{2\pi}(|X|^2-|Y|^2)(\gamma(s))ds\\
=&\int_0^{2\pi}
(\cos^2(s)-\sin^2(s))g_{22}(\gamma(s))ds
+\int_0^{2\pi}(\sin^2(s)-\cos^2(s))g_{33}(\gamma(s))ds\\
&+\int_0^{2\pi}4\sin(s)\cos(s)g_{23}(\gamma(s))ds\\
=&\int_0^{2\pi}\frac{1}{2}\frac{d}{ds}[\sin(2s)]
g_{22}(\gamma(s))ds-\int_0^{2\pi}\frac{1}{2}
\frac{d}{ds}[\sin(2s)]g_{33}(\gamma(s))ds\\
&-\int_0^{2\pi}\frac{d}{ds}[\cos(2s)]g_{23}(\gamma(s))ds\\
=&\frac{1}{L^2}\int_0^{2\pi}[x^2x^3(x^3g_{22,2}-x^2g_{22,3})]\circ \gamma ds+\frac{1}{L^2}\int_0^{2\pi}[x^2x^3(x^2g_{33,3}-x^3g_{33,2})]\circ \gamma ds\\
&+\frac{1}{L^2}\int_0^{2\pi}[((x^2)^2-(x^3)^2)(x^2g_{23,3}-x^3g_{23,2})]\circ \gamma ds,
\end{split}
\end{align}
where in the final line we have integrated by parts and used the double angle formulas to write $\sin(2s)$ and $\cos(2s)$ in terms of $x^2$ and $x^3$.
Combining this with \eqref{e:partgeocurv1} produces
\begin{align}\label{e:partgeocurv2}
\begin{split}
\int_{\Sigma_t^L\cap T_L}\frac{|X|}{|Y|^2}=&2\pi+\frac{1}{2L^3}\int_{\Sigma_t^L\cap T_L}x^2x^3[x^3g_{22,2}+x^2g_{33,3}]+O(L^{-2q})\\
&-\frac{1}{2L^3}\int_{\Sigma_t^L\cap T_L}x^2x^3[x^2g_{22,3}+x^3g_{33,2}]\\
&+\frac{1}{2L^3}\int_{\Sigma_t^L\cap T_L}[(x^2)^2-(x^3)^2](x^2g_{23,3}-x^3g_{32,2}).
\end{split}
\end{align}
We then have that \eqref{e:geocurv3} and \eqref{e:partgeocurv2} yield
\begin{align}
\begin{split}
\int_{\Sigma_t^L\cap T_L}\kappa_{t,L}=&2\pi+\frac{1}{2L^3}\int_{\Sigma_t^L\cap T_L}x^2x^3[x^3g_{22,2}+x^2g_{33,3}]\\
&-\frac{1}{2L^3}\int_{\Sigma_t^L\cap T_L}x^2x^3[x^2g_{22,3}+x^3g_{33,2}]\\
&+\frac{1}{2L^3}\int_{\Sigma_t^L\cap T_L}[(x^2)^2-(x^3)^2](x^2g_{23,3}-x^3g_{32,2})\\
&-\int_{\Sigma_t^L\cap T_L}\frac{1}{2L^3}(x^2(x^3)^2g_{22,2}+(x^2)^2x^3g_{33,3})\\
&+\int_{\Sigma_t^L\cap T_L}\left(\left[(x^2)^2x^3+\frac{1}{2}(x^3)^3\right]\frac{g_{22,3}}{L^3}+\left[(x^3)^2x^2+\frac{1}{2}(x^2)^3\right]\frac{g_{33,2}}{L^3}\right)\\
&-\int_{\Sigma_t^L\cap T_L}\left((x^2)^3\frac{g_{23,3}}{L^3}+(x^3)^3\frac{g_{32,2}}{L^3}\right)+O(L^{-1-q}+L^{-2q})\\
=&2\pi+\frac{1}{2L}\int_{\Sigma_t^L\cap T_L}(x^3g_{22,3}-x^3g_{32,2}+x^2g_{33,2}-x^2g_{23,3})+O(L^{-1-q}+L^{-2q}).
\end{split}
\end{align}
Finally, integrating over $[-L,L]$ gives the desired identity.
\end{proof}

\subsection{Proof of Theorem \ref{mass.bd}}

Recall that from \eqref{e:cptstern2} we have
\begin{equation}
\frac{1}{2}\int_{\Omega_L}\left(\frac{|\nabla^2 u|^2}{|\nabla u|}+R_g|\nabla u|\right)dV\leq  4\pi L-\int_{-L}^L\left(\int_{\Sigma_t^L\cap T_L}\kappa_{t,L}\right)dt+\int_{C_L}\partial_{\upsilon}|\nabla u|dA.
\end{equation}
On the other hand, it follows from Lemmas \ref{ddu.comp} and \ref{curv.comp} that
\begin{align}
\begin{split}
\int_{C_L}\partial_{\upsilon}|\nabla u| dA-\int_{-L}^L\left(\int_{\Sigma_t^L\cap T_L}\kappa_{t,L}\right)dt=&\frac{1}{2}\int_{D_L^{\pm}}\pm \sum_j (g_{1j,j}-g_{jj,1})dA\\
&+\frac{1}{2}\int_{T_L}\left[\frac{x^2}{L}(g_{21,1}-g_{11,2})+\frac{x^3}{L}(g_{31,1}-g_{11,3})\right]dA\\
&+\frac{1}{2}\int_{T_L}\left[\frac{x^2}{L}(g_{23,3}-g_{33,2})+\frac{x^3}{L}(g_{32,2}-g_{22,3})\right]dA\\
&-4\pi L+o(1)\\
=&-4\pi L+\frac{1}{2}\int_{C_L}\sum_j (g_{ij,j}-g_{jj,i})\upsilon^i dA+o(1),
\end{split}
\end{align}
Therefore
\begin{equation}
\frac{1}{2}\int_{\Omega_L}\left(\frac{|\nabla^2 u|^2}{|\nabla u|}+R_g|\nabla u|\right)dV\leq \frac{1}{2}\int_{C_L}\sum_j (g_{ij,j}-g_{jj,i})\upsilon^i dA+o(1),
\end{equation}
and taking the limit as $L\to\infty$ gives the desired inequality \eqref{masslowerb}.

Consider now the case of equality when $m=0$. From the arguments above, this implies that the harmonic coordinate function is linear $|\nabla^2 u|\equiv 0$ and that the Euler characteristic of the level sets is constant $\chi(\Sigma_t)=1$. In particular the boundary of the exterior region is empty $\partial N=\emptyset$, and thus $M\cong\mathbb{R}^3$. Since there are three linearly independent harmonic coordinate functions with $\nabla^2u\equiv 0$, the manifold is flat, yielding the isometry $(M,g)\cong(\mathbb{R}^3,\delta)$.

\bigskip
\noindent\textbf{Acknowledgements.} The authors would like to thank Dan Lee
for several helpful comments on an earlier version of this manuscript, in particular for pointing out the simplification in Remark \ref{rem1}.

\end{document}